\documentclass[preprint,1p]{elsarticle}

\makeatletter
 \def\ps@pprintTitle{%
 	\let\@oddhead\@empty
 	\let\@evenhead\@empty
 	\def\@oddfoot{\footnotesize\itshape
 		{} \hfill\today}%
 	\let\@evenfoot\@oddfoot
 }
\makeatother
\usepackage[unicode]{hyperref}
\usepackage[makeroom]{cancel}
 
\usepackage{soul}
\usepackage{tikz}

\usepackage{latexsym}
\usepackage{indentfirst}
\usepackage{amsxtra}
\usepackage{amssymb}
\usepackage{amsthm}
\usepackage{amsmath}

\usepackage{MnSymbol}
\usepackage{mathrsfs} 
\usepackage[scr=rsfs,cal=boondox]{mathalfa}

\usepackage{xcolor}
\usepackage{color}

\usepackage{amsfonts}

\usepackage[capitalise]{cleveref}
\usepackage{tikz-cd}
\usepackage{amscd}
\usepackage{pst-node}

\DeclareMathOperator{\Ann}{Ann}

\usepackage[capitalise]{cleveref}

\newtheorem{theor}{Theorem}[section]
\newtheorem*{theor*}{Theorem}
\newtheorem{prop}[theor]{Proposition}
\newtheorem{lemma}[theor]{Lemma}
\newtheorem{cor}[theor]{Corollary}
\newtheorem*{cor*}{Corollary}
\theoremstyle{definition}               
\newtheorem{defin}[theor]{Definition}
\newtheorem{ex}{Example}

\newtheorem{rem}[theor]{Remark}
\newtheorem{rems}[theor]{Remarks}

\newtheorem*{que*}{Question}

\newtheorem*{conv*}{Convention}
 \newtheorem{prob}{Problem}
 \newtheorem*{prob*}{Problem}
\newtheorem*{nota*}{Notation}

\DeclareMathOperator{\End}{End}
\DeclareMathOperator{\id}{id}

\DeclareMathOperator{\E}{E}

\begin{document}

\begin{frontmatter}

\title{Associative pentagon algebras}	

\author[unile]{Marzia MAZZOTTA
}
\ead{marzia.mazzotta@unisalento.it}
\address[unile]{Dipartimento di Matematica e Fisica “Ennio De Giorgi”,  Università 
del Salento, Via Provinciale Lecce-Arnesano, 73100 Lecce, Italy}


\author[uniwa]{Agata PILITOWSKA
}
\ead{agata.pilitowska@pw.edu.pl} 
 \address[uniwa]{Faculty of Mathematics and Information Science, Warsaw University of Technology, Koszykowa 75, 00-662 Warsaw, Poland}


\begin{abstract}
A set-theoretic solution to the Pentagon Equation can be described as a \emph{pentagon} algebra $(S, \cdot, \ast)$ such that $(S, \cdot)$ is a semigroup and the operations $\cdot$ and $\ast$ are related by two additional equations. 
This paper aims to investigate \emph{associative} pentagon algebras in which $(S, \ast)$ is also a semigroup.
We introduce and describe two families of associative pentagon algebras which are strongly determined by the properties of the semigroup $(S,\ast)$. We present a complete characterization of such algebras using semigroup equations. We also provide constructions of such associative pentagon algebras and give several classes of examples. 

\end{abstract}

\begin{keyword}
 pentagon equation \sep set-theoretic solution \sep varieties of semigroups \sep equations
 \MSC[2020]  16T25\sep 81R50 \sep 18B40 \sep 08B20 \sep 03C05
\end{keyword}
\end{frontmatter}

 \section*{Introduction}

The Pentagon Equation (PE) classically originates from the field of Mathematical Physics as it first appeared as an identity satisfied by the Racah coefficients \cite{Bie53}. It was also studied as a consistency relation in quasi-Hopf algebras \cite{Dr89}, in the context of conformal field theory \cite{MoSe89}, and in connection with three-dimensional integrable systems \cite{Mai90}. Recently, a specific class of matrices that participate in factorization problems that turn out to be equivalent to constant and entwining pentagon maps has been proposed in \cite{Kassot23}.
In general, PE has several applications and appears in different areas of mathematics, also with different terminologies in the literature. See, for instance, the following non-exhaustive list: \cite{Za92, BaSk93, Wo96, BaSk98, Mi98, JiLi05, Kawa10, Ka11, DoSe14, Kor24}. 
For more contexts in which the PE appears, refer the reader to the introduction of the paper on \emph{polygon equations} \cite{DiMu15}  along with the references therein.   In this regard, we also mention the more recent paper \cite{MuHo24}. 

If $V$ is a vector space, a linear map $S: V \otimes V \to V \otimes V$ is a \emph{solution to the PE} if $S_{12}S_{13}S_{23}=S_{23}S_{12}$, where $S_{ij}$ denotes the map $V \otimes V \otimes V \to V \otimes V \otimes V$ acting as $S$ on the $(i,j)$-th tensor factors and as the identity on the remaining factor. If $\tau$ denotes the flip operator on $V \otimes V$, then $S$ is a solution if and only if $R=\tau S$ satisfies the so-called \emph{reversed Pentagon Equation}, that is, $R_{23}R_{13}R_{12}=R_{12}R_{23}$.

As it is common for the Yang-Baxter equation, the problem of studying this equation can be reduced to an easier one. In fact, in 1998, Kashaev and Sergeev \cite{KaSe98} introduced a definition of a set-theoretic solution of
this equation. Specifically, if $X$ is a set and $s:X\times X\to X\times X$ is a map, the pair $(X, s)$ is a \emph{set-theoretic solution of the PE}, or briefly a \emph{solution}, if $s$ satisfies the relation
    $s_{23}s_{13}s_{12}=s_{12}s_{23}$,
	where $s_{ij}$ is the map $ X \times X \times X \to X \times X \times X$ acting as $s$ on the $(i, j)$-component and as $\id_X$ on the remaining one.
        Clearly, set-theoretic solutions induce solutions of the vector equation. 
        
        Writing the map $s$ as $s(x,y)=(x \cdot y, x \ast y)$, where $\cdot$ and $\ast$ are two binary operations on the set $X$, then one can check that $(X, s)$ is a solution if and only if for all $x,y,z \in X$ \begin{align*}
 & \, x\cdot(y\cdot z)=(x\cdot y)\cdot z,\notag\\
           &\left( x \ast y \right) \cdot \left(\left(x \cdot y\right) \ast z\right)  = x\ast \left(y \cdot z\right), \tag{I} \\
          & \left( x \ast y \right) \ast \left(\left(x \cdot y\right) \ast z\right)=  y \ast z. \tag{II}
        \end{align*}
We call such a triple $(X, \cdot, \ast)$ a \emph{pentagon algebra}.  In the pioneering paper \cite{KaSe98}, the focus was mainly on bijective solutions (that is, the map $s$ is bijective). In this vein, in \cite{KaRe07} invertible mappings arising from special Lie groups are also investigated. Furthermore, the paper \cite{CoJeKu20} contains a classification of involutive solutions, namely bijective solutions, in which $s^{-1}=s$. This result has recently been extended to the finite bijective case \cite{CoOkVAn24x}. 
The authors in \cite{CaMaMi19} studied not necessarily bijective solutions and achieved a classification of pentagon algebras when $(X, \cdot)$ is a group. In addition, in recent years, different classes of solutions have been studied and some classifications have been obtained; see \cite{CaMaSt20, MaPeSt24, Maz24, Castelli24x, EvKassTong24}. In this regard, \cite{Maz25} collects in a systematic way the results obtained so far on the solutions.

In this paper, we study \emph{associative pentagon algebras}, briefly \emph{APAs}, which are pentagon algebras such that $ \ast$ is also an associative operation. The problem of studying these algebras naturally arises by looking at the basic examples of solutions. It is enough to consider an arbitrary semigroup $(X, \cdot)$ and a right-zero semigroup $(X, \ast)$ to obtain a simple example of an APA. This is 
the only bijective solution if $(X, \cdot)$ is a group \cite{KaRe07}.

In any APA $(X,\cdot,\ast)$, for each $x\in X$, we can associate a left translation $\theta_x: X \to X$ by $x$ with respect to the operation $\ast$. 
      It directly follows from the associativity condition that the set
      $T(X):=\{\theta_x \, \mid \, x \in X\}$
     is a subsemigroup of the transformation monoid $X^X$, which we call the \emph{transformation semigroup} associated with $(X, \cdot, \ast)$.  
We establish that some properties of the APA can be reflected in its semigroup $(T(X), \circ)$ and vice versa.

We focus on two classes of APA that we call, respectively, \emph{left-trivially distributive} (LTD) and \emph{right-trivially distributive} (RTD). Specifically, an algebra $(X, \cdot, \ast)$ is 
\begin{center}
    \text{LTD \, if $\forall x,y,z \in X \, \,(x \cdot y) \ast z= x \ast z$;
    \qquad RTD \, if $\forall x,y,z \in X \,\,(x \cdot y) \ast z= y \ast z$.}
\end{center}
The intersection between the two classes is given by the APAs $(X, \cdot, \ast)$ in which $(X, \ast)$ is determined by an idempotent endomorphism $\gamma$ of $(X, \cdot)$, 
i.e., $x \ast y=\gamma(y)$, for all $x, y \in X$.

To study LTD or RTD associative pentagon algebras, we consider two families of semigroups. The first one is the variety $\mathcal{V}_{P_k}$ of right-normal $P_k$-semigroups. 
Recall that $(X, \ast)$ is a right-normal if, for all $x, y, z\in X$, $x \ast y \ast z=y\ast x\ast z$. In addition, for $0<k\in \mathbb{N}$, we say that $(X, \ast)$ is a \emph{$P_k$-semigroup} if, for all $x, y,z \in X$, \begin{align*}
x\ast y^k\ast z=x\ast z \tag{$P_k$}.
\end{align*}
In particular, $P_1$-semigroups are specifically Kimura semigroups (see \cite{Kimura, Agore}). 


The second one is the variety $\mathcal{V}_{Q_k}$ which contains \emph{$Q_k$-semigroups} $(S,\ast)$ satisfying, 
for all $x, y, z \in X$,
  \begin{align*}
  x \ast y^k \ast z= y \ast z.\tag{$Q_k$}
   \end{align*}

 The main results of the paper show that an APA $(X, \cdot, \ast)$ is LTD if and only if $(X, \ast) \in \mathcal{V}_{P_2}$ (Theorem \ref{thm:LTDAPA}) and it is RTD if and only if $(X, \ast) \in \mathcal{V}_{Q_1}$ (Theorem \ref{lemma_q1}).


Surprisingly, it turns out that some of the solutions investigated in \cite{CaMaSt20}, which are mappings satisfying both the PE and the quantum Yang-Baxter equation \cite{Dr92},  belong to the family of RTD APAs.

The paper is organized as follows. In Section 1, we give some examples and basic properties of associative pentagon algebras. In Section 2, we study the variety $\mathcal{V}_{P_k}$. We present a complete description of free $\mathcal{V}_{P_k}$-semigroups (Theorem \ref{thm:V2free}) and show that the variety $\mathcal{V}_{P_k}$ is locally finite, for any $k\in \mathbb{N}$. Section 3 is devoted to LTD APAs. We construct LTD APAs from the free $\mathcal{V}_{P_2}$-semigroups and show that for LTD APA $(X,\cdot,\ast)$, the transformation semigroup $(T(X), \circ)$ is commutative and $T(X)\subseteq \End(X, \cdot)$ (\cref{ts_comm}).    Moreover, we determine several examples of such APAs, including those for which $(T(X), \circ)$ is a monoid. The latter class is highlighted by the particular subvariety $\mathcal{V}_R$ of $\mathcal{V}_{P_2}$. In Section 4 we present several examples of APAs that are LTD. In particular, we show that all bijective APAs $(S,\cdot,\ast)$ are LTD and such that $(S,\ast)\in \mathcal{V}_R$. Similarly as for LTD algebras, in Section 5, we investigate RTD associative pentagon algebras and the variety $\mathcal{V}_{Q_1}$, with the relative constructions. Furthermore, we obtain that the transformation semigroup $(T(X), \circ)$ of RTD APAs is a right-zero semigroup (Proposition \ref{prop:lzs}). Even for these algebras, in Section 6 we provide various examples. Finally, in the last Section 7 we present an example of APA which is neither LTD nor RTD. We also give further perspectives on the future that can open research for other classes of APAs.

\smallskip

\section{Basics on associative pentagon algebras}
In this section,  we introduce the notion of associative pentagon algebras and provide some basic properties and examples that will be used throughout the paper.

\medskip

\begin{defin}
 A pentagon algebra $(S, \cdot, \ast)$ is called \emph{associative} if $(S, \ast)$ is a semigroup. 
\end{defin}
In the sequel, we shall briefly refer to any associative pentagon algebra as APA.  Moreover, we denote by $\E(S, \cdot )$ and $\E(S, \ast )$ the sets of idempotents of the semigroups $(S, \cdot)$ and $(S, \ast)$, respectively. 

\smallskip

The following are some simple examples of APA.

\begin{ex} 
   \cite{Mi04, CaMaSt20} Let $S$ be a set and consider two commuting maps $f, g:S \to S$ such that $f^2=f$ and $g^2=g$. 
   For $x, y \in S$, set $x \cdot y=f(x)$ and $x \ast y=g(y)$. Then $(S, \cdot, \ast)$ is an APA with $\E(S, \cdot)=\{x \in S \, \mid \, f(x)=x\}$ and $\E(S, \ast)=\{x \in S \, \mid \, g(x)=x\}$.
\end{ex}


\begin{ex}\label{ex_left}
    Let $S=\{a,b,c\}$ and $(S, \cdot)$ be the aperiodic semigroup defined as follows: $a \cdot x=a=c \cdot x$ and $b \cdot x=b$, for all $x \in S$. Moreover, consider the operations $\ast$ and $\ast'$ defined by:
$$\begin{tabular}{c | c c c c c}
    $\ast$ & $a$ & $b$ & $c$   \\
    \hline
   $a$& $a$ & $b$ & $a$ \\
    $b$ & $b$ & $a$ & $b$\\
    $c$& $a$ & $b$ & $a$
\end{tabular}  \qquad {\rm and}\qquad \begin{tabular}{c | c c c c c}
    $\ast'$ & $a$ & $b$ & $c$   \\
    \hline
   $a$& $b$ & $a$ & $b$ \\
    $b$ & $a$ & $b$ & $a$\\
    $c$& $b$ & $a$ & $b$
\end{tabular}.$$
Then $(S, \cdot, \ast)$ and $(S, \cdot, \ast')$ are APA.    

\end{ex}

\smallskip

 When it is convenient, we will use the following notation. 
\begin{nota*} Let $S$ be a non-empty set and for each $x \in S$, let $\theta_x\in S^S$.
 Set $x \ast y:= \theta_x(y)$, for all $x, y \in S$. If for each $x\in S$, $\theta_x=\gamma$, for some $\gamma\in S^S$,  we will say that $(S,\ast)$ is \emph{determined} by the mapping $\gamma$. Note that if $\gamma=\id_S$, then $(S,\ast)$ is a right-zero semigroup and if $\gamma=const\in S$, then $(S,\ast)$ is a zero-semigroup.
 
Then $(S, \cdot, \ast)$ is an APA if and only if $(S, \cdot)$ is a semigroup and the following hold
\begin{align}
    \theta_x(y) \cdot \theta_{x \cdot y}(z)&=\theta_x(y \cdot z) \label{p_one}\tag{I}\\
    \theta_{\theta_x(y)}\theta_{x \cdot y}&=\theta_y \label{p_two}\tag{II}\\
    \label{assoc} \theta_{\theta_x(y)}&=\theta_x\theta_y \tag{$\ast$}
\end{align}
for all $x,y, z\in S$. In particular, equation \eqref{assoc} is equivalent to requiring that $(S, \ast)$ be a semigroup. 
\end{nota*}

\smallskip

In the following, we denote by $\End{(S,\cdot)}$ the subsemigroup of the transformation monoid $S^S$ of $S$ of all endomorphisms of $(S,\cdot)$.

\begin{ex}\label{ex:gamma}
Let $(S,\cdot)$ be a semigroup, $\gamma\in S^S$ and $(S,\ast)$ be determined by $\gamma$. An algebra $(S,\cdot,\ast)$ is an APA if and only if $\gamma\in \End(S,\cdot)$ and $\gamma^2=\gamma$. 
In particular, it means that if $\gamma=const\in S$, then $const\in \E(S,\cdot)$.
\end{ex}

\begin{rem}
Clearly, not all pentagon algebras are associative. For example, take a group $G$ and define the operations $x \cdot y:=x$ and $x\ast y:=x^{-1}y$, for all $x,y \in G$. Then $(G, \cdot, \ast)$ is a pentagon algebra such that $\ast$ is not associative.
\end{rem}

\smallskip



Directly by \eqref{p_two} and \eqref{assoc}, we obtain a useful observation.
\begin{rem}
    Let $(S, \cdot, \ast)$ be an APA. Then for all $x, y\in S$:
      \begin{align}
        \theta_x=\theta_y\theta_x\theta_{y \cdot x}.\label{rel}
    \end{align} 
\end{rem}

\smallskip

Applying
\eqref{p_one} and \eqref{p_two}, the following immediately shows which semigroups are APA, as both operations can be considered equal.
\begin{lemma}\label{cdot=}
Let $(S,\cdot)$ be a semigroup. Then $(S,\cdot, \cdot )$ is an APA if and only if, for all $x,y,z \in S$, $x\cdot y\cdot z=y\cdot z$.
\end{lemma}

 \begin{ex} For any right-zero semigroup $(S, \cdot)$, the triple $(S,\cdot,\cdot)$ is an APA.
 \end{ex}

\begin{ex}\label{ex:freetwo}
Let $(S,\cdot,\cdot)$ be the APA such that $(S,\cdot)$ is a free semigroup satisfying the condition in \cref{cdot=} and generated by a set $X$.
Thus, we get $S=\{x, \,x\cdot y\mid x,y\in X\}$. 
\end{ex}

    

\smallskip

The following technical lemma will be crucial in proving some of the main results of the paper.

\begin{lemma}
Let $(S, \cdot, \ast)$ be an APA. Then, for all $x, y, z \in S$, 
\begin{align}\label{rel_4}
    \theta_{x \cdot y}\theta_z=\theta_y\theta_z\theta_x\theta_{y \cdot z}.
\end{align}
 \begin{proof}
     In fact, for $x, y, z\in S$, we have
    \begin{align*}
      \theta_{x \cdot y} \theta_z&\underset{\eqref{assoc}}{=}\theta_{\theta_{x \cdot y}(z)}\underset{\eqref{p_two}}{=}\theta_{\theta_{\theta_x(y)}\theta_{x \cdot y}(z)}\theta_{\theta_x(y) \cdot \theta_{x \cdot y}(z)}
      \\&
      \underset{\eqref{p_one}}{=}\theta_{\theta_{\theta_x(y)}\theta_{x \cdot y}(z)}\theta_{\theta_x(y \cdot z)}\underset{\eqref{p_two}}{=}\theta_{\theta_y(z)}\theta_{\theta_x(y \cdot z)}\underset{\eqref{assoc}}{=}\theta_y\theta_z\theta_x\theta_{y \cdot z},
    \end{align*}  
    which is the desired conclusion.
 \end{proof} 
\end{lemma}  

\smallskip

We observe that in any APA some additional properties on $(S,  \ast)$ also constrain the structure of the semigroup $(S, \cdot)$. For example,  the following result provides the classification of APAs in the case where $(S, \ast)$ is a group. 
\begin{prop}\cite[Proposition 17]{CaMaMi19} \label{s_ast group}
    Let $(S, \cdot, \ast)$ be a pentagon algebra such that $(S, \ast)$ is a group. Then $(S, \cdot)$ is a left-zero semigroup and the group $(S, \ast)$ is elementary $2$-abelian.
\end{prop}

\smallskip

The following shows that the only APA in which $(S, \ast)$ is a \emph{band}, i.e., an idempotent semigroup, is the one determined by $\gamma=\id_S$.

\begin{prop}\label{S_cdot_ast_idemp}
Let $(S,\cdot)$ be a semigroup and
$(S, \ast)$ be a band. Then $(S,\cdot,\ast)$ is an APA if and only if $(S, \ast)$ is determined by $\id_S$.
\begin{proof}
    First, by \eqref{assoc}, the maps $\theta_x$ are all idempotent. Further, for any $x,y\in X$, we get
    \begin{align*}
 \theta_x(y)=\theta_x\theta_y(y)\underset{\eqref{rel}}{=}\theta_x\theta_{x}\theta_y\theta_{x \cdot y}(y)=\theta_{x}\theta_y\theta_{x \cdot y}(y)\underset{\eqref{rel}}{=}\theta_y(y)=y,   
    \end{align*}
   which is the desired conclusion.    
   By Example \ref{ex:gamma}, the converse is obvious.
\end{proof}
\end{prop} 
\smallskip

As you will see in the next sections, to obtain classification of some APAs, it is convenient to start with the idempotents of $(S, \cdot)$. In fact, the above examples show that it is not possible to establish, in general, a precise link between the sets of idempotents of the two semigroups. 

\begin{ex}\hspace{1mm}
    \begin{enumerate}
    \item Let $(S,\cdot,\ast)$ be an APA from Example \ref{ex:gamma}. If $\gamma=\id_S$ then $\E(S,\ast)=S$ and $\E(S,\cdot)$ may be from the empty set to the hole $S$. If $\gamma=e\in \E(S,\cdot)$ then $\E(S,\ast)=\{e\}$ and the set $\E(S,\cdot)$ consists of at least one element.
    \vspace{2mm}
        \item For the APA $(S,\cdot,\ast)$ in \cref{cdot=}, there is 
        $\E(S, \cdot)=\E(S, \ast)$.
        In particular, for  the APA $(\{x, \,x\cdot y\mid x,y\in X\},\cdot,\cdot)$  from Example \ref{ex:freetwo}, $\E(S,\cdot)=\{x\cdot y\, \mid \, x,y\in X\}$.
\vspace{2mm}
\item For the APA $(S,\cdot,\ast)$ in \cref{s_ast group}, 
$\E\left(S, \cdot\right)=S$ and $\ \E\left(S, \ast\right)=\{1\}$.
  \end{enumerate}
  \end{ex}

\smallskip

\noindent In general, we can state the following.
\begin{prop}
 Let $(S, \, \cdot , \, \ast)$ be an APA. Then for all $x \in S$ and $e \in \E(S, \, \cdot)$ we have:
          $(x \cdot e) \ast e \in \E(S, \cdot)$.
 In particular, $\E(S, \, \cdot)  \cap \E(S, \, \ast)=\{e \ast e \, \mid \, e \in \E(S, \, \cdot)\}.$
 \begin{proof}
For $x \in S$ and $e \in \E(S, \, \cdot)$, by \eqref{p_one}, we have
        \begin{align*}
            \left((x \cdot e) \ast e \right) \cdot \left((x \cdot e) \ast e \right)= \left((x \cdot e) \ast e \right) \cdot \left((\left(x \cdot e) \cdot e\right) \ast e \right)=  (x\cdot e) \ast e.
        \end{align*}
Taking $x=e$, we get the claim.
    \end{proof}
\end{prop}

\smallskip

We now give some properties of the maps $\theta_x$ that involve the idempotents of $\E(S, \cdot)$ and will be helpful in classifying some APAs.
\begin{lemma}\label{lm:cubic}
   Let $(S, \cdot, \ast)$ be an APA. 
   Then for all $e \in \E(S, \cdot)$ and $x \in S$: 
   \begin{enumerate}
       \item $\theta_e\theta_e\theta_e=\theta_e$, 
       \item $\theta_e\theta_{e \cdot x}=\theta_{e\cdot x} \theta_{e \cdot x}$, 
       \item $\theta_{e \cdot x}\theta_e=\theta_{e}\theta_{e}$,
       \item $\theta_e=\theta_{e\cdot x\cdot e}$.
   \end{enumerate}   
   \begin{proof}
 The claim in $1.$ immediately follows by \eqref{rel}.
  Now, let $x\in S$ and $e \in \E(S, \cdot)$. By \eqref{rel} we have $\theta_{e \cdot x}=\theta_e\theta_{e \cdot x}\theta_{e \cdot e \cdot x}=\theta_e\theta_{e \cdot x}\theta_{ e \cdot x}$. Thus, since the map $\theta_e$ is cubic, we get
\begin{align*}
    \theta_e\theta_{e \cdot x}&=\theta_e\theta_e\theta_{e \cdot x}\theta_{ e \cdot x}=\theta_e\theta_e(\theta_e\theta_{e \cdot x}\theta_{ e \cdot x})(\theta_e\theta_{e \cdot x}\theta_{ e \cdot x})=(\theta_e\theta_{e \cdot x}\theta_{ e \cdot x})(\theta_e\theta_{e \cdot x}\theta_{ e \cdot x})=\theta_{e \cdot x}\theta_{e \cdot x},
\end{align*}
i.e., the equality in $2.$ follows. Moreover, 
\begin{align*}
\theta_{e \cdot x}\theta_e&\underset{\eqref{rel}}{=} \theta_{e \cdot x}(\theta_{e \cdot x}\theta_e\theta_{e \cdot x \cdot e}) \underset{2.}{=}\theta_{e }\theta_{e \cdot x}\theta_e\theta_{e \cdot  x \cdot e}\underset{\eqref{rel}}{=}\theta_{e}\theta_{e}
\end{align*}
and, finally, 
            $\theta_e\underset{\eqref{rel}}{=}\theta_{e\cdot x\cdot e}\theta_e\theta_{e\cdot x\cdot e}\underset{3.}{=}\theta_e\theta_e\theta_{e\cdot x\cdot e}\underset{2.}{=}\theta_e\theta_{e\cdot x\cdot e}\theta_{e\cdot x\cdot e}\underset{\eqref{rel}}{=}\theta_{e\cdot x\cdot e}$.
    \end{proof}
\end{lemma}  

\smallskip


\smallskip

As an immediate consequence, if $(S, \cdot)$ is a monoid, we have the following description.

\begin{prop}\label{monoid}
Let $(S, \cdot, \ast)$ be an APA such that $(S, \cdot)$ is a monoid with identity $1$. Then $(S,\ast)$ is determined by $\theta_1$. 
\begin{proof}
By Lemma \ref{lm:cubic}, for each $x\in X$, we obtain
$\theta_1\theta_x=\theta_x\theta_x$ and $\theta_x\theta_1=\theta_1\theta_1$.
Using \eqref{rel}, this implies that $\theta_x=\theta_1\theta_x\theta_x=\theta_1\theta_1\theta_x=\theta_x\theta_1\theta_x=\theta_1$.
\end{proof}
\end{prop}

\smallskip

\section{The variety of \texorpdfstring{$P_k$}{}-semigroups}
In this section, we investigate a special class of semigroups which we will later use to characterize some families of APA.

\medskip

Let $(S,\ast)$ be a semigroup and $n\in \mathbb{N}$. For any $x\in S$, let us denote by $x^n$ the  $n$th power of $x$ in $(S, \ast)$.

\begin{defin}
Let $0<k\in \mathbb{N}$. A semigroup  $(S,\ast)$ is called \emph{$P_k$-semigroup} if it satisfies the identity: for all $x,y,z \in S$  
    \begin{align}\label{P}
x\ast y^k\ast z=x\ast z \tag{$P_k$}.
\end{align}
\end{defin}

\begin{rems} \hspace{1mm}
\begin{enumerate}
 \item  $P_1$-semigroups are exactly Kimura semigroups \cite{Kimura, Agore}.
    \item  Observe that a $P_r$-semigroup also is a $P_{rk}$-semigroup, for all $k \in \mathbb{N}$. In particular, 
    if $k\mid m$ then a $P_k$-semigroup is also a $P_m$-semigroup. 
    \item If a semigroup is $P_k$-semigroup and $P_r$-semigroup then it is $P_m$-semigroup, where $m=\gcd(k,r)$. 
\end{enumerate}
\end{rems}

\begin{ex}\label{ex:1gen}
The semigroups $(\{a,b,c\},\ast)$ and $(\{a,b,c\},\ast')$ in \cref{ex_left} are commutative $P_2$-semigroups.
\end{ex}

\begin{ex}
Let $S$ be a non-empty set and let $\gamma\in S^S$ be an idempotent mapping. 
Then $(S,\ast)$ determined by $\gamma$ is a $P_k$-semigroup, for any $0<k\in \mathbb{N}$. In particular, each right-zero semigroup and each zero-semigroup is a $P_k$-semigroup, for any $0<k\in \mathbb{N}$. 
\end{ex}

\smallskip

\begin{nota*}
In the following, let $\mathcal{V}_{P_k}$ be the variety of all right-normal semigroups satisfying \eqref{P}. In this respect, we recall that a semigroup $(X,\ast)$ is \emph{right-normal} if $
x\ast y\ast z=y\ast x\ast z$, for all $x,y,z\in X$.
\end{nota*}

\smallskip

\begin{lemma}\label{lemmaP} 
    Let $(S,\ast) \in \mathcal{V}_{P_k}$.
    Then
    $x^k\ast z=y^k \ast z=z^{k+1},$
   for all $x,y,z \in S$.
    \begin{proof}
      For $x,y,z \in S$,  we have  
      $x^k\ast z\underset{\eqref{P}}{=}x^k\ast y^k\ast z= y^{k}\ast x^k \ast z
               \underset{\eqref{P}}{=}y^k\ast z.$
      Taking $z=y$, we obtain $x^k\ast z=z^{k+1}$.
    \end{proof}
\end{lemma}

\smallskip

Now we present a general construction of a right-normal  $P_k$-semigroup. \\
Let $X$ be a non-empty set of arbitrary cardinality and 
let $+_k$ denote the addition modulo $k$. 
Let $x_0\in X$ and $f, g \in \mathbb{Z}_k^{X}$. Let us define two new functions:
   \begin{align*}
   S[f,g;x_0](x)&:=\begin{cases}
       f(x)+_k g(x) &\text{if $x\neq x_0$} \\
        f(x_0)+_kg(x_0)+_k 1& \text{if $x=x_0$}
   \end{cases}, \qquad 
   T_{x_0}(x):=\begin{cases} 0
        &\text{if $x\neq x_0$} \\
        k& \text{if $x=x_0$}
   \end{cases}.
\end{align*} 
 Let us consider the set
\begin{align*}
&    \mathcal{A}_k(X):=\biggl\{ f\in \mathbb{Z}_k^{X}\biggr\} \times X\cup \bigcup_{x\in X} \biggl\{T_{x}\in \{0, k\}^{X}\biggr\} \times \{x\}.
\end{align*}
Define the following binary operation on the set $\mathcal{A}_k(X)$: for $(f,y),(g,z)\in \mathcal{A}_k(X)$ set
\begin{align*}
   (f,y) \, \ast \, (g,z):=
   \begin{cases}
       (S[f,g;y],z)  &\text{if} \,\, \exists x \in X \, S[f,g;y](x)\neq 0\\
        (T_{z},z)& \text{otherwise}
   \end{cases}.
\end{align*}

\smallskip

\begin{theor}\label{thm:VPK}
For any non-empty set $X$, the algebra  $\left(\mathcal{A}_k(X), \, \ast \right) \in \mathcal{V}_{P_k}$.
\end{theor} 
  \begin{proof}
 Let $(e,u), (f,y),(g,z)\in \mathcal{A}_k(X)$. Then we obtain
\begin{align*}
&L=\left(\left(e,u\right) \, \ast \, \left(f,y\right) \right)     \, \ast \, (g,z)=
\begin{cases}
       \left(S[e,f;u],y\right) \, \ast \, \left(g,z\right) &\text{if} \,\, \exists x \in X \, S[e,f;u](x)\neq 0\\
        (T_{y},y)\ast \, \left(g,z\right)& \text{otherwise}
   \end{cases}.
 \end{align*}
Let us observe that if $S[e,f;u](x)= 0$ for all $x\in X$, we  have 
\begin{align*}
&T_y(x)=\begin{cases}
      S[e,f;u](x)  &\text{if $x\neq y$}       \\
        S[e,f;u](y)+k& \text{if $x= y$}
   \end{cases}.
\end{align*} 
Hence, for any $a\in \mathbb{Z}$, 
\begin{align*}
&T_y(x)+_ka=S[e,f;u](x)+_ka \quad {\rm for}\; x\neq y, \quad {\rm and}\\
&T_y(y)+_ka=\left(S[e,f;u](y)+k\right)+_ka=S[e,f;u](y)+_ka \quad {\rm for}\; x= y.
\end{align*}
Further, 
 \begin{align*}
S[S[e,f;u],g;y](x)&
   &=\begin{cases}
       e(x)+_kf(x)+_kg(x)  &\text{if $x\neq u$ and $x\neq y$} \\
e(x)+_kf(x)+_k1+_kg(x)+_k 1 &\text{if
       $x=y=u$}\\
         e(u)+_kf(u)+_k1+_kg(u) &\text{if $x=u$ and $x\neq y$}      \\
        e(y)+_kf(y)+_kg(y)+_k1  &\text{if $x=y$ and $x\neq u$}
   \end{cases}   
\end{align*}

If there is $x\in X$ such that $S[S[e,f;u],g;y](x)\neq 0$ then $L=(S[S[e,f;u],g;y],z)$. Otherwise, $L=(T_z,z)$.
\\
\noindent
On the other hand,
 \begin{align*}
&R=\left(e,u\right) \, \ast \, \left(\left(f,y\right)      \, \ast \, (g,z)\right)=
\begin{cases}
       \left(e,u\right) \, \ast \,\left(S[f,g;y],z\right)
       &\text{if} \,\, \exists x \in X \, S[f,g;y](x)\neq 0\\
        \left(e,u\right)\ast (T_{z},z) & \text{otherwise}  \end{cases}.
 \end{align*}
 Similarly as previously, for any $a\in \mathbb{Z}$,
 \begin{align*}
 a+_kT_z(x)=a+_kS[f,g;y](x),\quad {\rm for}\; x\in X.
 \end{align*}
We have  
 \begin{align*}
S[e,S[f,g;y];u](x)&
   &=\begin{cases}
      e(x)+_kf(x)+_kg(x)  &\text{if $x\neq u$ and $x\neq y$}\\
      e(x)+_kf(x)+_kg(x)+_k 1+_k 1 &\text{if $x=y=u$}\\
        e(u)+_kf(u)+_kg(u)+_k1  &\text{if $x=u$ and $x\neq y$}       \\
        e(y)+_kf(y)+_kg(y)+_k1  &\text{if $x=y$ and $x\neq u$}  
   \end{cases}
\end{align*}
If there is $x\in X$ such that $S[e,S[f,g;y];u](x)\neq 0$, then $R=(S[e,S[f,g;y];u],z)$. Otherwise, $R=(T_z,z)$. 
Thus, $L=R$. Therefore, the operation $\ast$ is associative.
\vskip 2mm
By above calculations and since the addition $+_k$ is commutative, $\left(\mathcal{A}_k(X), \, \ast \right)$ is a right-normal semigroup. 

\vskip 3mm
\noindent
Now let us show that $\left(\mathcal{A}_k(X), \, \ast \right)$ is a $P_k$-semigroup. Note that 
$$(f,y)^k=(\underset{(k-1)-times}{\underbrace{S[\ldots S[S[S}}[f,f;y],f;y],f;y]\ldots,f;y],y)$$ and     
\begin{align*}
h(x):&=\underset{(k-1)-times}{\underbrace{S[\ldots S[S[S}}[f,f;y],f;y],f;y]\ldots,f;y](x)\\
&=\begin{cases}\underset{k-times}{\underbrace{f(x)+_k\ldots+_kf(x)}}=0
&\text{if $x\neq y$
}       \\
f(y)+_k\underset{(k-1)-times}{\underbrace{f(y)+_k1+_k\ldots+_kf(y)+_k1}}=k-1 &\text{if $x= y$}
      \end{cases}\, .
      \end{align*}
 Then
\begin{align*}
&S[e,h;u](x)=\begin{cases}
      e(x)  &\text{if  ($x\neq u$ 
      and $x\neq y$) or $x=u=y$}
              \\
        e(y)+_k(k-1)  &\text{if $x=y$ and $x\neq u$}
        \\
        e(u)+_k1  &\text{if $x=u$  and $x\neq y$}
   \end{cases}\, .
\end{align*}
If there is $x\in X$ such that $S[e,h;u](x)\neq 0$ then $\left(e,u\right)\,\ast\,\left(f,y\right)^k=\left(S[e,h;u],y\right)$. Otherwise, $\left(e,u\right)\;\ast\;\left(f,y\right)^k=(T_y,y)$. 
\\
\noindent
Further,
\begin{align*}
&S[S[e,h;u],g;y](x)=\begin{cases}
      e(x)+_kg(x)  &\text{if $x\neq u$ 
      and $x\neq y$
      }        \\
        e(y)+_k(k-1)+_kg(y)+_k1=e(y)+_kg(y)  &\text{if $x=y$ and $x\neq u$}
        \\
        e(u)+_k1+_kg(u)  &\text{if $x=u$}  
   \end{cases}
\end{align*}
Then, $S[S[e,h;u],g;y]=S[e,g;u]$. Moreover,  
\begin{align*}
&S[T_y,g;y](x)=\begin{cases}
      g(x)  &\text{if $x\neq y$ 
      }        \\
        k+_k1+_kg(y)=g(y)+_k1  &\text{if $x=y$}  
   \end{cases}.
\end{align*}
If $S[e,h;u](x)=0$ for all $x\in X$, we have $S[T_y,g;y]=S[e,g;u]$, which means that $\left(e,u\right)\,\ast\,\left(f,y\right)^k\,\ast\,\left(g,z\right)=\left(e,u\right)\,\ast\,\left(g,z\right)$.\\ 
Therefore, the claim follows.
  \end{proof} 
  
 \smallskip

 \begin{rem} We observe that for any $f\in \mathbb{Z}_k^X$ and $y\in X$, 
 $(f,y)^k=({\bf 0},y)^k$, 
 where ${\bf 0}\colon X\to \mathbb{Z}_k, \,  {\bf 0}(x):=0$.
 \end{rem}

\smallskip

\begin{lemma}
For the semigroup $\left(\mathcal{A}_k(X), \, \ast \right)$, $\E\left(\mathcal{A}_k(X), \, \ast \right)=\{({\bf 0},x)^k\mid x\in X\}$.
\end{lemma}
\begin{proof}
First note that 
\begin{align*}
   (T_y,y) \, \ast \, (T_y,y)= (S[T_y,T_y;y],y) \neq (T_y,y) 
\end{align*}
for any $y\in X$. Now, let $(f,y)\in \mathcal{A}_k(X)$ and $f\neq T_y$. Then 
    \begin{align*}
   (f,y) \, \ast \, (f,y)=
        (S[f,f;y],y) =(f,y) \quad \Leftrightarrow\quad S[f,f;y]=f.
\end{align*} 
Hence, $(f,y)$ is an idempotent element if and only if $f(y)=k-1$ and for $x\neq y$, $f(x)=0$. By the proof of \cref{thm:VPK} we obtain $f(x)=\underset{(k-1)-times}{\underbrace{S[\ldots S[S[S}}[f,f;y],f;y],f;y]\ldots,f;y](x).$ \end{proof}


\begin{ex}
Let $k=2$ and $X$ be a non-empty set of arbitrary cardinality. 
Let $x_0\in X$ and $f, g \in \mathbb{Z}_2^{X}$. Let us define two functions:
   \begin{align*}
   S[f,y;x_0](x):=\begin{cases}
       f(x)+_2 g(x) &\text{if $x\neq x_0$} \\
        f(x_0)+_2g(x_0)+_2 1& \text{if $x=x_0$}
   \end{cases},
   \qquad
   T_{x_0}(x):=\begin{cases} 0
        &\text{if $x\neq x_0$} \\
        2& \text{if $x=x_0$}
   \end{cases}
\end{align*} 
and let  
\begin{align*}
&    \mathcal{A}_2(X):=\biggl\{ f\in \mathbb{Z}_2^{X}\biggr\} \times X\cup \bigcup_{x\in X} \biggl\{T_{x}\in \{0, 2\bigr\}^{X}\biggr\} \times \{x\}.
\end{align*}
Defining on the set $\mathcal{A}_2(X)$ the following binary operation:  for $(f,y),(g,z)\in \mathcal{A}_2(X)$:
\begin{align*}
   (f,y) \, \ast \, (g,z):=
   \begin{cases}
       (S[f,g;y],z) &\text{if  $\exists \, x\in X$}       \;  \text{$S[f,g;y](x)=1$}
      \\
        (T_{z},z)& \text{otherwise}
   \end{cases},
   \end{align*}
we obtain $(\mathcal{A}_2(X),\ast)$ is a right-normal $P_2$-semigroup.
\end{ex}

 \medskip
 
  Now, let us consider the set \begin{align*}
      \displaystyle\mathcal{A}_kfin(X):=\{f_\alpha\in \mathcal{A}_k(X)\, \colon \displaystyle{\exists_{\alpha\subseteq X}}\; \alpha-\text{a finite set},\; \forall_{ x\in X\setminus \alpha}\; f(x)=0\}.
  \end{align*}
  It is easy to check that $\mathcal{A}_kfin(X)$ is a subsemigroup of $\left(\mathcal{A}_k(X), \, \ast \right)$. 
  Further, consider 
  $\mathcal{X}:=\{({\bf 0},x)\, \colon \, x\in X\}$.

\begin{theor}\label{thm:V2free}

For any non-empty set $X$, the algebra  $\left(\mathcal{A}_kfin(X), \, \ast \right)$ is free in the variety $\mathcal{V}_{P_k}$ over the set $\mathcal{X}$.
\end{theor}
\begin{proof}
 Clearly, the semigroup $\left(\mathcal{A}_kfin(X), \, \ast \right)$ is generated by the set $\mathcal{X}$. We will show that $\left(\mathcal{A}_kfin(X), \, \ast \right)$ has the universal mapping property for the variety  $\mathcal{V}_{P_k}$ over the set $\mathcal{X}$. Let $(S,\ast)\in  \mathcal{V}_{P_k}$ and $h\colon \mathcal{X}\to S$ be a mapping such that for any $({\bf 0},x)\in \mathcal{X}$, $h(({\bf 0},x))=s_x\in S$.
  
Let us define $\overline{h}\, \colon \, \mathcal{A}_kfin(X)\to S$ in the following way:
\begin{align*}
  \overline{h}(({\bf 0},x)):=s_x\quad {\rm and}\quad   \overline{h}\left(\left(f_\alpha,y\right)\right):=\prod_{x\in \alpha}s_x^{f_\alpha(x)}\ast s_y,\; {\rm if}\; f_{\alpha}\neq {\bf 0}.
\end{align*}
We will show that $\overline{h}$ is a semigroup homomorphism. 
For $(f_\alpha,y),(g_\beta,z)\in \mathcal{A}_kfin(X)$ we have
\begin{align*} \overline{h}&\left(\left(f_\alpha,y\right)\ast\left(g_\beta,z\right)\right)=
    \begin{cases}
\overline{h}\left(\left(S[f_\alpha,g_\beta;y],z\right)\right) &\displaystyle{\exists_{x\in \alpha\cup\beta} \;S[f_\alpha,g_\beta;y](x)\neq 0}
      \\
        \overline{h}((T_{z},z))& \text{otherwise}
   \end{cases}
   \\
   &=
   \begin{cases}
       \displaystyle\ \prod_{x\neq y\, \in \,\alpha\cup \beta}s_x^{f_\alpha(x)+_kg_\beta(x)}\ast s_y^{f_\alpha(y)+_kg_\beta(y)+_k1} \ast s_z &\text{if}\; \exists_{x\in \alpha\cup\beta} \;S[f_\alpha,g_\beta;y](x)\neq 0\\ 
       s_z^{k+1}& \text{otherwise}
   \end{cases}
   \end{align*}
   On the other hand,
   \begin{align*}
\overline{h}\left(\left(f_\alpha,y\right)    \right)\ast\overline{h}\left(\left(g_\beta,z\right)\right)=
       \prod_{x\in \alpha}s_x^{f_\alpha(x)}\ast s_y\ast
       \prod_{x\in \beta}s_x^{g_\beta(x)}\ast s_z.
   \end{align*}
   By right-normal law and Lemma \ref{lemmaP} we obtain:
   \begin{align*}
       \prod_{x\in \alpha}&s_x^{f_\alpha(x)}\ast s_y\ast
       \prod_{x\in \beta}s_x^{g_\beta(x)}\ast s_z=
       \displaystyle \prod_{x\in \alpha\cup\beta}s_x^{f_\alpha(x)+_kg_\beta(x)}\ast s_y\ast s_z\\
       &        =\begin{cases}
       \displaystyle\prod_{x\neq y\, \in \, \alpha\cup \beta}s_x^{f_\alpha(x)+_kg_\beta(x)}\ast s_y^{f_\alpha(y)+_kg_\beta(y)+_k1} \ast s_z &\text{if}\;  \displaystyle{\exists_{x\in \alpha\cup\beta} \;S[f_\alpha,g_\beta;y](x)\neq 0}\\
          s_z^{k+1}& \text{otherwise}
   \end{cases}
   \end{align*}
   Hence, $\overline{h}\left(\left(f_\alpha,y\right)\ast\left(g_\beta,z\right)\right)=\overline{h}\left(\left(f_\alpha,y\right)\right)\ast\overline{h}\left(\left(g_\beta,z\right)\right)$. By definition, 
$\overline{h}_{/\mathcal{X}}=h$. Since the set $\mathcal{X}$ generates the semigroup $\left(\mathcal{A}_kfin(X), \, \ast \right)$, the homomorphism $\overline{h}$ is uniquely defined, which completes the proof.
  \end{proof}

\smallskip

\begin{cor}
Each semigroup $(S,\ast)\in \mathcal{V}_{P_k}$ is a homomorphic image of $\left(\mathcal{A}_kfin(X), \, \ast \right)$, for sufficiently large $X$.   
\end{cor}

\medskip

If a set $X$ is finite with $|X|=n\in \mathbb{N}$, the construction reduces to the following one.
Let us consider the set
\begin{align*}
    \mathcal{A}_k(n):&=\biggl\{ (a_1, \dots, a_n, l) \mid a_i \in \mathbb{Z}_k \,  \text{for} \, \,i \in [n], \, l \in [n] \biggr\} \bigcup\biggl\{  (0, \dots, \underset{m}{k}, \dots, 0, m)\mid  m \in [n] \; \biggr\}
\end{align*}
and define the following binary operation: for $(a_1, \dots, a_n, l), (b_1, \dots, b_n, m)\in \mathcal{A}_k(n)$ let us set
\begin{align*}
   (a_1, \dots, a_n, l) \, \ast \, (b_1, \dots, b_n, m):=\begin{cases}
       (s_1, \dots, t_l, \dots, s_n, m) &\text{if there exists $j\in [n]$, $j \neq l$} \\
      \qquad  &\text{such that $s_j\neq 0$ or $t_l\neq 0$}\\
       (0, \dots, \underset{m}{k}, \dots, 0, m) & \text{otherwise}
   \end{cases},
\end{align*}
where for each $i\in [n]$, 
    $s_i:=a_i+_k b_i$ and  $t_i:=a_i+_kb_i+_k 1$.
    
   \smallskip
   
\begin{theor}
For any $k,n\in \mathbb{N}$, the algebra  $\left(\mathcal{A}_k(n), \, \ast \right) \in \mathcal{V}_{P_k}$ and has $nk^n + n$ elements.
\end{theor}

\begin{cor}
For any $k\in \mathbb{N}$, the variety $\mathcal{V}_{P_k}$ is locally finite.
\end{cor}

\smallskip

\begin{ex}
For $n=1$, we have $\mathcal{A}_2(1)=\{(0, 1), (1, 1), (2, 1)\}$. Then
$$\begin{tabular}{c | c c c c c}
    $\ast$ & $(0,1)$ & $(1,1)$ & $(2,1)$   \\
    \hline
   $(0,1)$& $(1,1)$ & $(2,1)$ & $(1,1)$ \\
    $(1,1)$ & $(2,1)$ & $(1,1)$ & $(2,1)$\\
    $(2,1)$& $(1,1)$ & $(2,1)$ & $(1,1)$
    \end{tabular}
 $$   
  and the  algebra $(\mathcal{A}_2(1),\ast)$ is a $P_2$-semigroup isomorphic with the first semigroup from \cref{ex:1gen}. The only idempotent element is $(1,1)$.
 \end{ex}

 \smallskip
 
 \begin{ex} 
For $n=2$, we have $\mathcal{A}_2(2)=\{(0, 0, 1), (1, 0, 1), (0, 1, 1), (1, 1, 1), (2, 0, 1), \\(0,0,2),(0, 1, 2), (1, 0, 2), (1, 1, 2), (0, 2, 2) \}$. For example, 
\begin{align*}
& (0,0,1) \, \ast \, (0,1,2)= (1, 1, 2)\quad {\rm and}\quad
(0,1,2) \, \ast\,  (0,0,1) = (0, 2, 1).
    \end{align*}
    In this case $(\mathcal{A}_2(2),\ast)$ is not commutative and $\E(\mathcal{A}_2(2),\ast)=\{(1,0,1),(0,1,2)\}$.
 \end{ex}

\smallskip

\subsection{Subvariety \texorpdfstring{$\mathcal{V}_R$}{}}
\smallskip

\begin{nota*}
    Let $\mathcal{V}_R$ be the variety of semigroups $(S,\ast)$ which satisfy the following identity:
\begin{align}\label{R}
x^2\ast z=z \tag{R}.
\end{align}
\end{nota*}

\begin{rem}
    Note that \eqref{R} implies $y\ast x^{2k}\ast z=y\ast z$, for all $x,y,z\in S$ and $k \in \mathbb{N}$. Hence $\mathcal{V}_R$ is a subvariety of $\mathcal{V}_{P_{2k}}$, for $k \in \mathbb{N}$. 
\end{rem}

\smallskip

\begin{ex}\label{ex:R}
Let $n=1$ and $\mathcal{C}_1=\{x:=(0, 1), x^2:=(1, 1)\}$. Then $(\{x,x^2\},\ast)$ with
$$\begin{tabular}{c | c c c cc}
    $\ast$ & $x$ & $x^2$    \\
    \hline
   $x$& $x^2$ & $x$  \\
    $x^2$ & $x$ & $x^2$ 
    \end{tabular}
 $$ 
 is a right-normal semigroup which satisfies the property \eqref{R}.
\end{ex}

\medskip

In general, let $X$ be a set and let us consider the set 
$\mathcal{C}(X):=\{ f\in \mathbb{Z}_2^{X}\} \times X$.
Define the binary operation on the set $\mathcal{C}(X)$ as follows: for $(f,y),(g,z)\in \mathcal{C}(X)$ set
\begin{align*}
   (f,y) \, \ast \, (g,z):=
       (S[f,g;y],z) .
\end{align*}
Straightforward calculations show that $\left(\mathcal{C}(X), \, \ast \right)$ is a  right-normal semigroup satisfying \eqref{R}. 
In particular, $(\mathcal{C}fin(X):=\{f_\alpha\in \mathcal{C}(X)\, \colon \exists _{\alpha\subseteq X} \; \alpha-\text{a finite set},\; \forall_{ x\in X\setminus \alpha}\; f(x)=0\},\ast)$ is a subsemigroup of $\left(\mathcal{C}(X), \, \ast \right)$.
Similarly as previously, we can prove the following.
\begin{theor}
For any non-empty set $X$, the algebra  $\left(\mathcal{C}fin(X), \, \ast \right)$ is free in the variety $\mathcal{V}_{R}$ over the set $\mathcal{X}$.
\end{theor}

If $|X|=n \in \mathbb{N}$ set
$\mathcal{C}_n:=\mathbb{Z}_2^n\times [n]$ 
and for $(a_1, \dots, a_n, k), (b_1, \dots, b_n, m)\in \mathcal{C}_n$, let us define the following binary operation: 
\begin{align*}
   (a_1, \dots, a_n, k) \, \ast \, (b_1, \dots, b_n, m):=\left(a_1+_2 b_1 \dots, a_k+_2b_k+_2 1, \dots, a_n+_2b_n, m\right). 
\end{align*}

\begin{theor}
   $\left(\mathcal{C}_n, \, \ast \right)$ has  $n2^n$-elements and is a 
  $\mathcal{V}_R$-semigroup.
\end{theor}

\smallskip

\section{Left-trivially distributive APAs}
In this section, we introduce and study a class of algebras to which some APAs belong. In particular, we show that for such APAs, $(S, \ast)$ necessarily belongs to the variety of $P_{2k}$-semigroups. Moreover, we introduce the transformation semigroup of left translations associated with an APA and show how some properties of the APA may be reflected in its semigroup and vice versa.

\begin{defin}
   An algebra $(S,\cdot,\ast)$ is called \emph{left-trivially distributive} if, for $x,y,z\in S$, the following law holds: \begin{align}\label{LTD}\tag{LTD}
        (x\cdot y)\ast z=x\ast z.
    \end{align}
    \end{defin}

\begin{ex}
    Let $(S, \cdot)$ be a left-zero semigroup and $(S, \ast)$ be a magma. Then the algebra $(S, \cdot, \ast)$ is left-trivially distributive. 
\end{ex}

\smallskip

      Recall that for an APA $(S,\cdot,\ast)$, for each $x\in X$, $\theta_x$ is a left translation by $x$ with respect to the operation $\ast$. 
      It directly follows from condition \eqref{assoc} that the set 
      $$T(S):=\{\theta_x \, \mid \, x \in S\}$$ 
     is a subsemigroup of the transformation monoid $S^S$.
     \begin{defin}
          We call $(T(S),\circ)$ the \emph{transformation semigroup of left translations associated with} an APA $(S, \cdot, \ast)$.
     \end{defin}     

\smallskip

\begin{rem}
Let $(T(S),\circ)$ be the transformation semigroup associated with an APA $(S, \cdot, \ast)$. Then $(T(S),\circ)$ is commutative if and only if the semigroup $(S,\ast)$ is right-normal.
\end{rem}
      






\smallskip
  
\begin{lemma} \label{ts_comm}
 Let $(S,\cdot,\ast)$ be a left-trivially distributive APA.
Then $T(S)\subseteq \End{(S,\cdot)}$ and $(T(S),\circ)$ is commutative.
 \begin{proof}
 Clearly, condition \eqref{LTD} is equivalent to $\theta_{x\cdot y}=\theta_x$, for $x,y, \in S$. Thus, by \eqref{p_one}, this means that $\theta_x$ is an endomorphism of $(S, \cdot)$.
  In addition, by \eqref{rel_4}, for any $x,y\in S$, we have $\theta_x\theta_y=\theta_{x \cdot x}\theta_y\underset{\eqref{rel_4}}{=}\theta_x\theta_y\theta_x\theta_x=\theta_x\theta_y\theta_{x \cdot y}\theta_x\underset{\eqref{rel}}{=}\theta_y\theta_x$. Therefore, $(T(S),\circ)$ is commutative.
 \end{proof}
   \end{lemma}

\begin{rem}
  Recall that a solution $(S, s)$ is called \emph{commutative} if $s_{12}s_{13}=s_{13}s_{12}$; \emph{co-commutative}  if $s_{23}s_{13}=s_{13}s_{23}$ (see \cite[Definition 6]{CaMaSt20}). It is an easy computation to check that, writing the solution $s$ as usual, then $(S, s)$ is commutative if and only if, 
  $(S, \cdot, \ast)$ is left-trivially distributive and $(S, \cdot)$ is a left-normal semigroup.  Instead, $(S, s)$ is cocommutative if and only if, for $x,y,z \in S$, $x \cdot y=x \cdot\theta_z(y)$ and $\theta_x\theta_y=\theta_y\theta_x.$

Hence, if $(S,s)=(S,\cdot, \ast)$ is a co-commutative APA, then $(S, \ast)$ is a right-normal semigroup. Moreover, by \cref{ts_comm} 
, if $(S,s)=(S, \cdot, \ast)$ is commutative, then $(S, \cdot)$ is left-normal, while $(S, \ast)$ is right-normal. 
\end{rem}

   \begin{ex}
  Let $(S,\cdot,\ast)$ be a commutative or co-commutative APA. Then its transformation semigroup $(T(S),\circ)$ is commutative. 
   \end{ex}

   \smallskip

   Recall that a semigroup $(X,\ast)$ is cancellative if, for all $x,y,x\in X$, $ x\ast y=x\ast z$ implies $x=y$ and $y\ast x=z\ast x$ also implies $y=z$. 
\begin{lemma}\label{lm:cancel}
 Let $(S,\cdot,\ast)$ be an APA such that the transformation semigroup $(T(S),\circ)$ is cancellative. Then $(S,\cdot,\ast)$ is left-trivially distributive.
 \begin{proof}
Let $x, y, z \in S$.  By \eqref{rel_4}, 
we have:
  $\theta_y\theta_x\theta_x\theta_{y\cdot x}\underset{\eqref{rel_4}}{=}\theta_{x\cdot y}\theta_x\underset{\eqref{rel}}{=}\theta_{x\cdot y}\theta_y\theta_x\theta_{y\cdot x}$,
   which implies
\begin{align}
\theta_y\theta_x=\theta_{x\cdot y}\theta_y.\label{eq:xyyyx}
\end{align}  
Now, by cancellativity and again by \eqref{rel_4}, 
we get:
  \begin{align*}
\theta_y\theta_y\theta_x\theta_{y\cdot y}\underset{\eqref{rel_4}}{=}\theta_{x\cdot y}\theta_y\underset{\eqref{eq:xyyyx}}{=}\theta_y\theta_x\underset{\eqref{rel}}{=}\theta_y\theta_y\theta_x\theta_{y \cdot x}\quad \Rightarrow\quad \theta_{y\cdot x}=\theta_{y\cdot y}.
 \end{align*}
 Hence, by \eqref{eq:xyyyx} for $y=x$,  and by cancellativity we obtain: 
\begin{align*}
    \theta_{x\cdot x}\theta_x=\theta_x\theta_x\quad \Rightarrow\quad \theta_x=\theta_{x\cdot x}=\theta_{x\cdot y}.
\end{align*}
  \end{proof}
   \end{lemma}
   
\smallskip

As a consequence of \cref{ts_comm} and \cref{lm:cancel}, we have the following.
\begin{cor}\label{cor_can}
Let $(S,\cdot,\ast)$ be an APA. Then each cancellative transformation semigroup $(T(S),\circ)$ is commutative.       
   \end{cor}

\begin{ex}\label{expk}
The transformation semigroup of any APA $(S,\cdot,\ast)$ with $(S,\ast)\in \mathcal{V}_{P_k}$ is commutative and cancellative. Indeed, by \eqref{P} for $x,y,z\in S$
\begin{align*}
  \theta_x\theta_z=\theta_y\theta_z\quad \Rightarrow\quad   \theta_x\theta_z^k=\theta_y\theta_z^k\quad \Rightarrow\quad \theta_x=\theta_y,
\end{align*}
hence $(T(S),\circ)$ is cancellative. Moreover, by \cref{cor_can}, $(T(S),\circ)$ is also commutative. 
\end{ex}

\smallskip

Clearly, by \cref{lm:cancel},  any APA $(S,\cdot,\ast)$ with $(S,\ast)\in \mathcal{V}_{P_k}$ is left-trivially distributive. Note that we can reverse this statement for specific $k$. 

\begin{theor}\label{thm:LTDAPA}
An APA $(S,\cdot,\ast)$ is left-trivially distributive if and only if the semigroup $(S,\ast)\in \mathcal{V}_{P_{2n}}$, with $n \in \mathbb{N}$.  
\end{theor}
\begin{proof}
    First, assume that $(S,\cdot,\ast)$ is LTD. By \cref{ts_comm}, the transformation semigroup $ (T(S),\circ)$ is commutative.  So, by \eqref{rel}, for $x,y\in S$ we have:
      $\theta_x=\theta_y\theta_x\theta_{y\cdot x}=\theta_y\theta_x\theta_{y}=\theta_x\theta_y^2.$
  Thus, by induction arguments, $\theta_x=(\theta_x\theta_y^2)\theta_y^2=\ldots =\theta_x\theta_y^{2n}$, with $n \in \mathbb{N}$. Therefore, $(S,\ast)\in \mathcal{V}_{P_{2n}}$.
    The converse follows from \cref{lm:cancel} together with \cref{expk}.
\end{proof}
 
 \smallskip

\begin{prop}\label{prop:APAV2} 
   Let $(S, \cdot, \ast)$ be a LTD algebra such that $(S,\cdot)$ and $(S,\ast)$ both are semigroups. Then $(S,\cdot,\ast)$ is an APA if and only if $(S,\ast)\in \mathcal{V}_{P_2}$ and $T(S)\subseteq \End{(S,\cdot)}$.
    \begin{proof}
Let $(S,\ast)\in \mathcal{V}_{P_2}$ and $T(S)\subseteq \End{(S,\cdot)}$ and let $x, y, z \in S$. Then, the equation \eqref{p_one} follows by 
    $\left( x \ast y \right) \cdot \left(\left(x \cdot y\right) \ast z\right) = \left( x \ast y \right) \cdot \left(x \ast z\right) =x\ast \left(y \cdot z\right)$.          
Moreover, since $(S, \ast)$ is right-normal, we have
     $x \ast y \ast \left(x \cdot y\right) \ast z=  x \ast y  \ast x \ast z=y \ast x^2  \ast z =y \ast z$.
Hence, equation \eqref{p_two} is also satisfied. 
The converse is true by \cref{thm:LTDAPA} and \cref{ts_comm}.
    \end{proof}
\end{prop}

\smallskip

\subsection{Construction of left-trivially distributive APAs}
Now we describe a general method on how to obtain an APA from any semigroup $(S,\ast)\in \mathcal{V}_{P_k}$.

\medskip

\begin{lemma}\label{lm:form}
    Let $(S,\cdot,\ast)$ be a left-trivially distributive APA generated by a set $X\subseteq S$. Let $\langle X\rangle_\ast$ be the subsemigroup of $(S,\ast)$ generated by $X$. Then for each $w\in S$ there exist $s_1,\ldots,s_p\in \langle X\rangle_\ast$ such that $w=s_1\cdot s_2\cdot \ldots\cdot s_p$.

\begin{proof}
  The proof goes by induction on the minimal number $\ell$ of occurrences of the semigroup operation $\cdot$ in the expression $w$ as a word in the alphabet $X$.

  Consider $w=s_1$ with $s_1\in \langle X\rangle_\ast$. Hence, the result holds for $\ell=0$. Now suppose that the hypothesis is established for $\ell>0$ and let $w\in S$ be an element in which the semigroup operation $\cdot$ occurs $\ell+1$ times. 

  $i$) Let $w=s_1\cdot s_2$ for some $s_1,s_2\in \langle X\rangle_\ast$. According to the induction hypothesis, there are $s_{11},\ldots,s_{1p},s_{21},\ldots,s_{2r}\in \langle X\rangle_\ast$ such that $s_1=s_{11}\cdot\ldots\cdot s_{1p}$ and $s_2=s_{21}\cdot \ldots\cdot s_{2r}$. Therefore, $w=s_1\cdot s_2=s_{11}\cdot\ldots\cdot s_{1p}\cdot s_{21}\cdot \ldots\cdot s_{2r}$.

  $ii$) Otherwise, there are $s_1,s_2\in \langle X\rangle_\ast$ such that $w=s_1\ast s_2$ and $s_1=t_1\cdot t_2$ or $s_2=t_1\cdot t_2$, for some $t_1,t_2\in \langle X\rangle_\ast$. Then, by the left-trivially distributivity, $w=s_1\ast s_2=(t_1\cdot t_2)\ast s_2=t_1\ast s_2$ or by the left-distributivity, $w=s_1\ast(t_1\cdot t_2)=(s_1\ast t_1)\cdot (s_2\ast t_2)$, which completes the induction proof.
\end{proof}
\end{lemma}

\smallskip

Let $(S,\ast)\in \mathcal{V}_{P_2}$ and $(B(S),\cdot)$ be a semigroup generated by $S$. Hence, for each $x\in B(S)$ there are $x_1,\ldots,x_n\in S$ such that $x=x_1\cdot x_2\cdot\ldots\cdot x_n=\prod_{i=1}^nx_i$. For $x= \displaystyle\prod_{i=1}^nx_i$ and $y= \displaystyle\prod_{j=1}^my_j\in B(S)$, with $x_1,\ldots,x_n,y_1,\ldots,y_m\in S$, define the following binary operation:
\begin{align*}
    x\otimes y:=\prod_{j=1}^m(x_1\ast y_j).
\end{align*}
\begin{prop}
    $(B(S),\cdot,\otimes)$ is a left-trivially distributive APA.
    \begin{proof}
  Let $x=\displaystyle\prod_{i=1}^nx_i$, $y=\displaystyle\prod_{j=1}^my_j$ and $z=\displaystyle\prod_{k=1}^pz_k$, with $x_1,\ldots,x_n,y_1,\ldots,y_m,z_1,\ldots,z_p\in S$. In the first place, we will show that $(B(S),\otimes)$ is a semigroup. 
  \begin{align*}
      &L:=(x\otimes y)\otimes z=\prod_{j=1}^m(x_1\ast y_j)\otimes \prod_{k=1}^pz_k=\prod_{k=1}^p((x_1\ast y_1)\ast z_k),\quad {\rm and}\\
&R:=x\otimes(y\otimes z)=\prod_{i=1}^nx_i\otimes \prod_{k=1}^p(y_1\ast z_k)=\prod_{k=1}^p(x_1\ast (y_1\ast z_k)).
  \end{align*}
  Since the operation $\ast$ is associative, $L=R$ and $(B(S),\otimes)$ is a semigroup.
  Further, by assumption, $(S,\ast)\in \mathcal{V}_{P_2}$, then
  \begin{align*}
      x\otimes y\otimes z=\prod_{k=1}^p(x_1\ast y_1\ast z_k)=\prod_{k=1}^p(y_1\ast x_1\ast z_k)=y\otimes x\otimes z,
  \end{align*}
  which shows that $(B(S),\otimes)$ is right-normal.

  By associativity of the operation $\otimes$ and fact that $(B(S),\otimes)$ is a $P_2$-semigroup we obtain
  \begin{align*}
      x\otimes y\otimes y\otimes z=\prod_{k=1}^p(x_1\ast y_1\ast y_1\ast z_k)=\prod_{k=1}^p(x_1\ast z_k)=\prod_{i=1}^nx_i\otimes \prod_{k=1}^pz_k=x\otimes z.
  \end{align*}
  Finally, we prove that $(B(S),\cdot,\otimes)$ is left-trivially distributive and the operation $\cdot$ is left-distributive over $\otimes$. We have:
  \begin{align*}
      (x\cdot y)\otimes z&=(\prod_{i=1}^nx_i\cdot \prod_{j=1}^my_j)\otimes \prod_{k=1}^pz_k=\prod_{k=1}^p(x_1\ast z_k)=\prod_{i=1}^nx_i\otimes \prod_{k=1}^p z_k=x\otimes z,\quad {\rm and}\\
      x\otimes (y\cdot z)&=\prod_{i=1}^nx_i\otimes(\prod_{j=1}^my_j\cdot \prod_{k=1}^pz_k)=\prod_{j=1}^m(x_1\ast y_j)\cdot \prod_{k=1}^p(x_1\ast z_k)=\\
      &=(\prod_{i=1}^nx_i\otimes \prod_{j=1}^my_j)\cdot (\prod_{i=1}^nx_i\otimes \prod_{k=1}^p z_k)=(x\otimes y)\cdot (x\otimes z).
  \end{align*}
  By Proposition \ref{prop:APAV2}, $(B(S),\cdot,\otimes)$ is an APA, which finishes the proof.
    \end{proof}
\end{prop}

\medskip

Let $X$ be a non-empty set, $\left(\mathcal{A}_2fin(X), \, \ast \right)$ the free semigroup in the variety $\mathcal{V}_{P_2}$ over the set $\mathcal{X}$, $\mathcal{S}$ a variety of semigroups $(S,\cdot)$, and let $F_{\mathcal{S}}\left(\mathcal{A}_2fin(X)\right)$ be the free semigroup in $\mathcal{S}$ generated by the set $\mathcal{A}_kfin(X)$. 
\begin{theor}\label{th:LTDconst}
    The algebra $(F_{\mathcal{S}}\left(\mathcal{A}_2fin(X)\right),\cdot,\otimes)$ is free in the variety of all left-trivially distributive APAs $(S,\cdot,\ast$) over the set $\mathcal{X}$, such that $(S,\cdot)\in \mathcal{S}$.
   \begin{proof}
   Let $\mathcal{SP}_{2}$ be the variety of all APAs $(S,\cdot,\ast)$ such that $(S,\ast)\in \mathcal{V}_{P_2}$ and $(S,\cdot)\in \mathcal{S}$. Let $(C,\cdot,\ast)\in \mathcal{SP}_{2}$ and $h\colon \mathcal{X}\to C$ be a mapping such that for each $x\in X$, $h(({\bf 0}, x))=c_x$. Since $(\mathcal{A}_2fin(X),\ast)$ is free in the variety $\mathcal{V}_2$, we can extend $h$ to a semigroup homomorhism $\overline{h}\colon (\mathcal{A}_2fin(X),\ast)\to (C,\ast)$. Now, since  $F_{\mathcal{S}}\left(\mathcal{A}_2fin(X)\right)$ is free in the variety $\mathcal{S}$ we can extend $\overline{h}$ to the semigroup homomorphism
$\overline{\overline{h}}\colon (F_{\mathcal{S}}(\mathcal{A}_2fin(X)),\cdot)\to (C,\cdot)$. Hence, for $a=\prod_{i=1}^na_i$ with $a_i\in \mathcal{A}_2fin(X)$, we have
$\overline{\overline{h}}(a)=\overline{\overline{h}}(\prod_{i=1}^na_i)=\prod_{i=1}^n\overline{\overline{h}}(a_i)=
\prod_{i=1}^n\overline{h}(a_i)$. 
We will show that $\overline{\overline{h}}$ is also a homomorphism of the algebras $ (\mathcal{A}_2fin(X),\otimes)$ and $(C,\ast)$.
Let $a=\prod_{i=1}^na_i$ and $b=\prod_{j=1}^mb_j$ with $a_i, b_j\in \mathcal{A}_2fin(X)$. 
\begin{align*}
\overline{\overline{h}}(a\otimes b)&=
\overline{\overline{h}}(\prod_{j=1}^m(a_1\ast b_j))=
\prod_{j=1}^m\overline{h}(a_1\ast b_j)=
\prod_{j=1}^m(\overline{h}(a_1)\ast \overline{h}(b_j))\\
&=\prod_{i=1}^ n\overline{h}(a_i)\otimes 
\prod_{j=1}^m\overline{h}(b_j)=
\overline{\overline{h}}(a)\otimes \overline{\overline{h}}(b).
\end{align*}
Moreover, let $x\in \mathcal{X}$. Then 
$\overline{\overline{h}}(({\bf 0},x))=\overline{h}(({\bf  0},x))=c_x$ which implies 
$\overline{\overline{h}}_{/\mathcal{X}}=h$ and,
completes the proof.
   \end{proof}
\end{theor}

\smallskip

 \begin{cor}
   $\left(S(\mathcal{C}(X)), \, \cdot,\, \otimes \right)$ is an APA such that $\left(S(\mathcal{C}(X)), \, \otimes \right)$ is a right-normal semigroup satisfying \eqref{R}. Moreover, the algebra $(F_{\mathcal{S}}\left(\mathcal{C}fin(X)\right),\cdot,\otimes)$ is free in the variety of all APAs $(S,\cdot,\ast$) over the set $\mathcal{X}$ such that $(S,\cdot)\in \mathcal{S}$ and $(S,\ast)\in \mathcal{V}_R$.
\end{cor}

\smallskip

\begin{ex}
 Let $(\mathcal{C}_1,\ast)=(\{x,x^2\},\ast)$ be the $\mathcal{V}_R$-semigroup from Example \ref{ex:R} and consider the free band $\left(S(\mathcal{C}_1),\cdot\right)$ generated by the set $\{x,x^2\}$. Then we have that $$S(\mathcal{C}_1)=\{x,x^2,x\cdot x^2, x^2\cdot x, x\cdot x^2\cdot x, x^2\cdot x\cdot x^2\}.$$ On the set $S(\mathcal{C}_1)$ define the following binary operation:
 $$\begin{tabular}{c | c c c c c c}
    $\otimes$ & $x$ & $x^2$ & $x\cdot x^2$  &$x^2\cdot x$& $x\cdot x^2\cdot x$& $x^2\cdot x\cdot x^2$\\
    \hline
    $x$& $x^2$ & $x$ & $x^2\cdot x$  &$x\cdot x^2$& $x^2\cdot x\cdot x^2$& $x\cdot x^2\cdot x$ \\
   $x^2$& $x$ & $x^2$ & $x\cdot x^2$  &$x^2\cdot x$& $x\cdot x^2\cdot x$& $x^2\cdot x\cdot x^2$ \\
    $x\cdot x^2$ & $x^2$ & $x$ & $x^2\cdot x$  &$x\cdot x^2$& $x^2\cdot x\cdot x^2$& $x\cdot x^2\cdot x$ \\
    $x^2\cdot x$& $x$ & $x^2$ & $x\cdot x^2$  &$x^2\cdot x$& $x\cdot x^2\cdot x$& $x^2\cdot x\cdot x^2$ \\
    $x\cdot x^2\cdot x$& $x^2$ & $x$ & $x^2\cdot x$  &$x\cdot x^2$& $x^2\cdot x\cdot x^2$& $x\cdot x^2\cdot x$ \\
    $x^2\cdot x\cdot x^2$& $x$ & $x^2$ & $x\cdot x^2$  &$x^2\cdot x$& $x\cdot x^2\cdot x$& $x^2\cdot x\cdot x^2$ \\
    \end{tabular}
 $$   
Then $\left(S(\mathcal{C}_1),\cdot,\otimes\right)$ is an APA with $\theta_{x^2}=\theta_{x^2\cdot x}=\theta_{x^2\cdot x\cdot x^2}=\id$ and $\theta_{x}=\theta_{x\cdot x^2}=\theta_{x\cdot x^2\cdot x}\neq\id$ but it is a self-inverse bijection.
 \end{ex}

\smallskip


\section{Some classes of left-trivially distributive APAs}
In this section, we provide some families of left-trivially distributive APAs.

\smallskip

 \subsection{APA \texorpdfstring{$(S,\cdot,\ast)$}{} such that \texorpdfstring{$(T(S), \circ)$}{} is a monoid}

\medskip

\begin{lemma}\label{lm:eq}
Let $(S,\cdot,\ast)$ be an APA. The following conditions are equivalent:
\begin{enumerate}
\item[(i)] There is $x\in S$ such that $\theta_x\theta_x=\id_S$.
\item[(ii)] There is $a\in S$ with $\theta_a=\id_S$.
\item[(iii)] For each $x\in S$, $\theta_x\theta_x=\id_S$.
\item[(iv)] There is $a\in S$ such that $\theta_a$ is a bijection.
\end{enumerate}
\begin{proof}
Let $s\in S$ be such that $\theta_s\theta_s=\id$. Then, by \eqref{assoc}, it follows that there is $a\in S$ such that $\theta_s\theta_s=\theta_a=\id$. By \eqref{rel_4} 
we obtain for any $y\in S$:
\begin{align*}
\theta_{a\cdot y}=\theta_{a\cdot y}\theta_a=\theta_y\theta_a\theta_{a}\theta_{y\cdot a} =
\theta_y\theta_a\theta_{y\cdot a}\underset{\eqref{rel}}{=}\theta_a=\id_S.
\end{align*}
Once again, by \eqref{rel_4}, for $y=z=a$ we obtain for any $x\in S$:
$\theta_{x\cdot a}=\theta_x\theta_{a\cdot x}=\theta_x$,
which, by the equation \eqref{rel}, implies 
$\theta_x\theta_x=\theta_x\theta_{x\cdot a}=\theta_x\theta_a\theta_{x\cdot a}=\theta_a=\id_S$,
for any $x\in S$. Finally, let $\theta_a$ be a bijection. Then 
we have
\begin{align*}
   \theta_a \underset{\eqref{rel}}{=}\theta_a\theta_a\theta_{a\cdot a}\quad \Rightarrow \quad \id_S= \theta_a\theta_{a\cdot a}\underset{\eqref{assoc}}{=}\theta_{\theta_a(a \cdot a)}\in T(S).
\end{align*}  
\end{proof}
\end{lemma}

\noindent
 Directly by \cref{lm:eq} we get the following.
\begin{cor}\label{cor:group}
If the transformation semigroup $(T(S),\circ)$ associated with an APA $(S, \cdot, \ast)$ is a monoid, then it is an elementary abelian $2$-group and $(S,\ast)\in \mathcal{V}_R$.
\end{cor}

\smallskip
\begin{theor}\label{prop_R}
Let $(S,\cdot,\ast)$ be an APA such that for some $a\in S$, $\theta_a$ is a bijection. 
 Then $(S,\cdot,\ast)$ is left-trivially distributive.
\begin{proof}
By \cref{cor:group}, $(T(S),\circ)$ is cancellative and in consequence, by \cref{lm:cancel}, it is left-trivially distributive.
\end{proof}
\end{theor}

\smallskip



If $(S, \cdot)$ is a group and $(S, \ast)$ is a right-zero semigroup, the solution $(S, s)$ associated with the APA $(S, \cdot, \ast)$ is bijective \cite{KaRe07}. We wonder if other APAs give rise to bijective solutions.
\begin{cor}
    Let $(S,\cdot,\ast)$ be an APA such that $S$ is finite and the associated solution $(S, s)$ is bijective. 
 Then $(S,\cdot,\ast)$ is left-trivially distributive and $(S, \cdot)$ is a left group.
\begin{proof}
Let $(S,s)$ be a bijective solution associated with APA $(S,\cdot,\ast)$. By \cite[Proposition 2.4]{CoOkVAn24x} $(S, \cdot)$ is a left group and by \cite[Proposition 2.10]{CoOkVAn24x}, for each $x\in S$, $\theta_x$ is a bijection. 
Therefore by \cref{prop_R} the claim follows.
\end{proof}
\end{cor}

\begin{prob}
    Do they exist othere APAs that give rise to bijective solutions in the infinite case?
\end{prob}

\smallskip


\smallskip

\subsection{APA \texorpdfstring{$(S,\cdot,\ast)$}{} such that \texorpdfstring{$(S, \cdot)$}{} has a right identity}

\medskip

We consider APAs $(S, \cdot, \ast)$ in which $(S, \cdot)$ has a right identity, i.e., there exists $e\in S$ such that $x \cdot e=x$, for all $x \in S$.

\begin{lemma}\label{lemma_commute}
     Let $(S,\cdot, \ast)$ be an APA and assume that $e\in S$ is a right identity for $(S, \cdot)$. Then for all $x\in S$:
          \begin{enumerate}
                  \item $\theta_{e\cdot x}=\theta_e$,
                           \item $\theta_e\theta_x=\theta_x\theta_e$,
                           \item $\theta_e^2\theta_x=\theta_x$.
     \end{enumerate}
     
    \begin{proof}
Clearly, $e\in \E(S,\cdot)$. By \cref{lm:cubic}-4. we have $\theta_{e\cdot x}=\theta_e$. Further, we obtain:
\begin{align*}
    \theta_e\theta_x=\theta_{e\cdot e} \theta_x\underset{\eqref{rel_4}}{=}\theta_e\theta_x\theta_e\theta_{e\cdot x}\underset{\ref{lm:cubic}}{=}\theta_e\theta_x\theta_{e\cdot x}\theta_{e}\underset{\eqref{rel}}{=}\theta_x\theta_{e} \quad \Rightarrow\quad \theta_e^2\theta_x=\theta_e\theta_x\theta_{e\cdot x}\underset{\eqref{rel}}{=}\theta_x.
\end{align*}
 which is our claim.
    \end{proof}
\end{lemma}

\smallskip

\begin{prop}\label{left identity}
       Let $(S,\cdot, \ast)$ be an APA and assume that $(S, \cdot)$ admits a right identity.
       Then the transformation semigroup $(T(S),\circ)$ is commutative.
       \begin{proof}
       Let $e\in S$ be a right identity of $(S,\cdot)$. By \cref{lemma_commute}, for $x, y \in S$ we obtain
       \begin{align*}
        \theta_x  \theta_y=\theta_{x\cdot e}\theta_y\underset{\eqref{rel_4}}{=}\theta_e\theta_y\theta_x\theta_{e\cdot y}=\theta_e\theta_y\theta_x\theta_e=\theta_e^2\theta_y\theta_x=\theta_y\theta_x. 
       \end{align*}
       \end{proof}
\end{prop}


\begin{theor}\label{thm:leftanih}
      Let $(S,\cdot, \ast)$ be an APA and assume that $(S, \cdot)$ admits a right identity. Then $(S,\cdot, \ast)$ is left-trivially distributive.
      \begin{proof}
      Let $e\in S$ be a right identity of $(S,\cdot)$. For $x \in S$, we have $\theta_e\underset{\eqref{rel}}{=}\theta_x\theta_e\theta_x \underset{\ref{left identity}}{=}\theta_x^2\theta_e$.
      Further, for $x,y\in S$  
         \begin{align*}
             \theta_{x \cdot y}&\underset{\ref{lemma_commute}}{=}\theta_{x \cdot y}\theta_e\theta_e\underset{\eqref{rel_4}}{=} \theta_y\theta_e\theta_x\theta_{y \cdot e}\theta_e  \underset{\ref{left identity}}{=} \theta_y^2\theta_e\theta_e\theta_x= \theta_e\theta_e\theta_x\underset{\ref{lemma_commute}}{=}\theta_x. 
         \end{align*}
        Therefore, the claim follows.
      \end{proof}
\end{theor}

\smallskip

 \subsection{APA \texorpdfstring{$(S,\cdot,\ast)$}{} such that \texorpdfstring{$(S, \cdot)$}{} has a left annihilator}

\medskip
 
 We consider APA $(S, \cdot,\ast)$ such that $(S, \cdot)$ has left annihilators. We briefly recall that a \emph{left annihilator} is an element $L\in S$ such that $L \cdot x=L$, for all $x \in S$. Similarly, a \emph{right annihilator} $R$ can be defined as an element $R \in S$ such that $x \cdot R=R$, for all $x \in S$. Moreover, an \emph{annihilator} of $(S,\cdot)$ is an element $0 \in X$ such that $0$ is both a left and a right annihilator. Clearly, left and right annihilators are elements of $E(S,\cdot)$.
 

\medskip


In the following, for a semigroup $(S, \cdot)$, let us introduce the set:
\begin{align*}
\Ann_l(S, \cdot):=\{ L \in S \, \mid \, 
\text{$L$ is a left annihilator in $(S, \cdot)$}\}.
\end{align*}



\begin{lemma}\label{left_ann}
    Let $(S,\cdot, \ast)$ be an APA and $L\in \Ann_l(S, \cdot)$. Then for all $x \in S$:
    \begin{enumerate}
        \item $\theta_x=\theta_L\theta_x\theta_L$,
        \item $\theta_x=\theta_x\theta_L\theta_L=\theta_L\theta_L\theta_x$,
        \item $\theta_x=\theta_{x \cdot L}$.
         \item $\theta_x\theta_x\theta_x=\theta_x$.
             
    \end{enumerate}
        \begin{proof}
            Let $x \in S$. The claim in $1.$ easily follows from \eqref{rel} since $\theta_x=\theta_L\theta_x\theta_{L\cdot x}=\theta_L\theta_x\theta_L$.  Obviously, $x\cdot L,L\in \E(S,\cdot)$, then by \cref{lm:cubic}, $\theta_{x \cdot L}$ and $\theta_L$ are cubic. 
      Next,  we get
$\theta_x=\theta_L\theta_x\theta_L=\theta_L\theta_L(\theta_L\theta_x\theta_L)=\theta_L\theta_L\theta_x.$
 Analogously, one can prove that $\theta_x=\theta_x\theta_L\theta_L$.       
 Moreover, we have
$\theta_{x \cdot L}\underset{1.}{=}  \theta_L\theta_{x \cdot L}\theta_L \underset{\eqref{rel_4}}{=} \theta_L(\theta_L\theta_L\theta_x\theta_{L}) \underset{1.}{=} \theta_x.$ 
        \end{proof}
\end{lemma}

\smallskip

\begin{theor}\label{teo_left_ann}
  Let $(S,\cdot, \ast)$ be an APA such that $\Ann_l(S, \cdot)$ is non-empty. Then $(S, \cdot, \ast)$ is left-trivially distributive. 
\end{theor}
  \begin{proof}
 
  First, we will show that the transformation semigroup $(T(S),\circ)$ is commutative. Let $x,y\in S$ and $L\in \Ann_l(S, \cdot)$. Then by \eqref{rel_4} and \cref{left_ann} we have:
 \begin{align*}
    \theta_x\theta_y& \underset{\ref{left_ann}}{=} \theta_{x \cdot L}\theta_y\underset{\eqref{rel_4}}{=} \theta_L\theta_y\theta_x\theta_{L \cdot y}
     \underset{\ref{left_ann}}{=}\theta_y\theta_x.
 \end{align*}
  Hence, 
    $ \theta_y\underset{\eqref{rel}}{=}\theta_{x}\theta_y\theta_{x \cdot y}\underset{\ref{left_ann}}{=}\theta_{x}\theta_y\theta_L\theta_{x\cdot y}\theta_L\underset{\eqref{rel_4}}{=}\theta_{x}\theta_y\theta_L\theta_y\theta_L\theta_x\theta_{y\cdot L}\underset{\ref{left_ann}}{=}\theta_x^2\theta_{y}^3\theta_L^2x=\theta_x^2\theta_{y}$.
 This implies that $(S,\ast)\in \mathcal{V}_{P_2}$ and by \cref{thm:LTDAPA} we obtain the desired conclusion.
  \end{proof}

\smallskip

\begin{ex} The APA in \cref{ex_left} is such that $\Ann_l(S, \cdot)=\{a, b\}$ and the algebras $\left(S, \cdot, \ast\right)$ and $\left(S, \cdot, \ast'\right)$ are left-trivially distributive. 
\end{ex}

\smallskip

\begin{cor} \label{prop_zero}
    Let $(S,\cdot, \ast)$ be an APA such that $(S, \cdot)$ has an annihilator $0$. Then $(S,\ast)$ is determined by 
    $ \theta_0 \in \End(S,\cdot)$.
    \begin{proof}
By \cref{teo_left_ann}, we have $\theta_x=\theta_{x \cdot 0}=\theta_0$, 
for all $x \in S$. Thus, the claim follows by the equation \eqref{p_two}. 
    \end{proof}
\end{cor}

\smallskip

\section{Right-trivially distributive APAs}
In this section, we introduce and investigate the ``mirror class" of right-trivially distributive APAs. As you will see, they are quite different from the left-trivially distributive ones. In this case, we also give a variety of semigroups that will be helpful to describe it.

\medskip

\begin{defin}\label{def:RTD}
   An algebra $(S,\cdot,\ast)$ is called \emph{right-trivially distributive} if for all $x,y,z\in S$  the following law holds:
    \begin{align*}\label{RTD}\tag{RTD}
        (x\cdot y)\ast z=y\ast z.
    \end{align*}
    \end{defin}


        \begin{ex} \hspace{1mm}
        \begin{enumerate}
            \item     Let $(S, \cdot)$ be a right-zero semigroup and $(S, \ast)$ be a magma. Then the algebra $(S, \cdot, \ast)$ is right-trivially distributive. 
            \item The APAs $(S, \cdot, \ast)$ as in \cref{cdot=} are right-trivially distributive.
        \end{enumerate}
\end{ex}


\smallskip

The following shows which APAs are both left and right-trivially distributive.
\begin{prop}\label{prop:both}
   An APA $(S, \cdot, \ast)$ is both left and right-trivially distributive if and only if $(S, \ast)$ is determined by some $\gamma \in \End(S, \cdot)$.
   \begin{proof}
   Assume that $(S, \cdot, \ast)$ is both left and right-trivially distributive. Then, for all $x, y \in S$, $\theta_x=\theta_{x \cdot y}=\theta_y$. From \eqref{p_two}, the claim follows.
   
     For the converse, evidently, the APAs in \cref{ex:gamma} are both left-trivially distributive and right-trivially distributive.   
   \end{proof} 
\end{prop}

\begin{ex}
Let $n \geq 1$, $S=\{0,1,\ldots,n-1\}$, and define for $x,y\in S$ two binary operations:
\begin{align*}
x\cdot y=\begin{cases}n-2  &{\rm if}\quad y=n-2\\
x  &{\rm if}\quad x=y\\
0  &{\rm if}\quad x\neq y\; \,{\rm and}\; \,y\neq n-2
\end{cases}
\quad \text{and} \quad 
x\ast y=\begin{cases}n-1  &{\rm if}\quad y= n-1\\
0  &{\rm if}\quad y\neq n-1
\end{cases}\, .
\end{align*}
Clearly, for all $x,y\in S$, $\theta_x=\theta_y=\theta$. Further, it is not difficult to check that for $y,z\in S$:
\begin{align*}
 &\theta(y)\cdot \theta(z)=
 \begin{cases}
 0 \quad &{\rm if} \quad y\neq n-1 \quad {\rm or}\quad z\neq n-1\\
 n-1 \quad &{\rm if}\quad x= y=n-1
 \end{cases}\quad =\quad\theta(y\cdot z).
 \end{align*}
Then $\theta\in \End(S, \cdot)$ and $(S,\cdot,\ast)$ is a left and right-trivially distributive APA.
\end{ex}

\smallskip

Let us now introduce the next family of semigroups that will be helpful to describe some new classes
of APAs.
    \begin{defin}
Let $0<k\in \mathbb{N}$. A semigroup $(S,\ast)$ is called \emph{$Q_k$-semigroup} if it satisfies for all $x, y, z \in S$
  \begin{align}\label{Q}
  x \ast y^k \ast z= y \ast z.\tag{$Q_k$}
   \end{align}

\end{defin}

  \noindent Note that if $(S,\ast)$ is a $Q_r$-semigroup, then $(S,\ast)$ is a $Q_{rk}$-semigroup, for all $k \in \mathbb{N}$. Moreover,  $Q_1$-semigroups are known in the literature as \emph{exclusive semigroups} (see \cite{Yamada}). 

\smallskip

\begin{nota*}
    In the following, let $\mathcal{V}_{Q_k}$ denote the variety of all semigroups satisfying \eqref{Q}.
\end{nota*}

\begin{ex}
Let $X$ be a set and assume that $0\notin X$. Let $\mathcal{B}(X):=(\{0\}\cup X)\times X$. Let us define a binary operation $\ast\,\colon \mathcal{B}(X)\times \mathcal{B}(X)\to \mathcal{B}(X)$ in the following way: for $(a,b), (c,d)\in \mathcal{B}(X)$
\begin{align*}
(a,b)\ast (c,d) :=\begin{cases}(c,d) &{\rm if}\quad c\neq 0\\
(b,d) &{\rm if}\quad c=0
\end{cases}
\end{align*}
It is easy to check that $(\mathcal{B}(X),\ast)$ is a semigroup generated by the set $\mathcal{X}:=\{(0,x)\, \colon \, x\in X\}$ and $(\mathcal{B}(X),\ast)\in \mathcal{V}_{Q_1}$. Clearly, $\E(\mathcal{B}(X),\ast)=\{(a,b)\mid a\neq 0\}$.
\end{ex}

\smallskip

\begin{ex}\label{ex:Q}
Let $\mathcal{B}(\{x\})=\{x:=(0, 1), \, x^2:=(1, 1)\}$. Then $(\{x,x^2\},\ast)$ with
$$\begin{tabular}{c | c c c cc}
    $\ast$ & $x$ & $x^2$    \\
    \hline
   $x$& $x^2$ & $x^2$  \\
    $x^2$ & $x^2$ & $x^2$ 
    \end{tabular}
 $$ 
 is a semigroup which satisfies the property ($Q_1$).
\end{ex}
\smallskip
\begin{theor}
For any non-empty set $X$, the algebra $\left(\mathcal{B}(X), \, \ast \right)$ is free in the variety $\mathcal{V}_{Q_1}$ over the set $\mathcal{X}$. 
   \begin{proof}
   Let $(S,\ast)\in \mathcal{V}_{Q_1}$ and $h\colon \mathcal{X}\to S$, $(0,x)\mapsto s_x\in S$. Note that $\overline{h}\,\colon \mathcal{B}(X)\to S$, 
   \begin{align*}
\overline{h}((a,b)):=\begin{cases}s_a\ast s_b &{\rm if}\quad a\neq 0\\
s_b &{\rm if}\quad a=0
\end{cases}
\end{align*}
is a semigroup homomorphism. Indeed, for $(a,b), (c,d)\in \mathcal{B}(X)$, by $(Q_1)$ we have
\begin{align*}
\overline{h}((a,b)\ast(c,d))=\begin{cases}\overline{h}((c,d))&{\rm if}\quad c\neq 0\\
\overline{h}((b,d)) &{\rm if}\quad c=0
\end{cases}\quad =\begin{cases}s_c\ast s_d&{\rm if}\quad c\neq 0\\
s_b\ast s_d &{\rm if}\quad c=0
\end{cases}
\end{align*}
Moreover, 
\begin{align*}
\overline{h}((a,b))\ast\overline{h}((c,d))=\begin{cases}\overline{h}((a,b))\ast s_c\ast s_d&{\rm if}\quad c\neq 0\\
\overline{h}((a,b))\ast s_d &{\rm if}\quad c=0
\end{cases}\quad \underset{(Q_1)}{=}\begin{cases}s_c\ast s_d&{\rm if}\quad c\neq 0\\
s_b\ast s_d &{\rm if}\quad c=0
\end{cases}
\end{align*}
Clearly, $\overline{h}/_{\mathcal{X}}=h$ which finishes the proof.
   \end{proof}
   \end{theor}

\smallskip
\begin{cor}
    The variety $\mathcal{V}_{Q_1}$  is locally finite.
\end{cor}
\smallskip
\begin{theor}\label{lemma_q1}
Let $(S,\cdot,\ast)$ be an APA. Then $(S,\cdot,\ast)$ is right-trivially distributive if and only if $(S, \ast) \in \mathcal{V}_{Q_1}$.
\begin{proof}
   Assume that $(S,\cdot,\ast)$ is right-trivially distributive.  By \eqref{rel}, for each $x, y \in S$
    \begin{align}\label{rtd}
\theta_y=\theta_x\theta_y\theta_y.    \end{align}
Therefore, for $x, y \in S$, we have
           $$\theta_x\theta_y=\theta_{\theta_y(y) \cdot x}\theta_y\underset{\eqref{rel_4}}{=}\theta_x\theta_y\theta_{\theta_y(y)}\theta_{x \cdot y}\underset{\eqref{assoc}}{=}\theta_x\theta_y\theta_y\theta_y\theta_y\underset{\eqref{rtd}}{=}\theta_x\theta_y\theta_y\underset{\eqref{rtd}}{=}\theta_y.$$
        Hence, $(S, \ast) \in \mathcal{V}_{Q_1}$.

        Vice versa, if $(S, \ast) \in \mathcal{V}_{Q_1}$, for $x, y \in S$ we have $\theta_y\underset{\eqref{rel}}{=}\theta_x\theta_y\theta_{x \cdot y}=\theta_y\theta_{x \cdot y}=\theta_{x \cdot y }$, which is our claim.
\end{proof}
\end{theor}
 It directly follows from \cref{lemma_q1}:
\begin{cor}\label{prop:lzs}
  Let $(S, \cdot, \ast)$ be an APA. Then $(S, \cdot, \ast)$ is right-trivially distributive if and only if $ (T(S),\circ)$ is a right-zero semigroup. In particular, in any right-trivially distributive APA, all the maps $\theta_x$, for $x\in S$, are idempotent.
  \end{cor}

\smallskip


\begin{ex}
    Let $(S, \cdot, \ast)$ be an APA that also is a solution of the quantum Yang-Baxter equation (QYBE) \cite{Dr92}. Then by \cref{lemma_q1} and by the condition Y2 in \cite[Proposition 8]{CaMaSt20}, $(S, \cdot, \ast)$ is right-trivially distributive.
    \end{ex}

\smallskip

\smallskip

 \begin{prop}\label{prop:rtd}
    Let $(S, \cdot, \ast)$ be a RTD algebra such that $(S,\cdot)$ and $(S,\ast)$ both are semigroups. Then $(S, \cdot, \ast)$ is an APA if and only if $(S, \ast) \in V_{Q_1}$ and for all $x, y, z\in S$.
    \begin{align}\label{false_distri}
        x\ast \left(y \cdot z\right)=\left( x \ast y \right) \cdot \left( y \ast z\right),
    \end{align}
    \begin{proof}
       Initially, assume that $(S, \cdot, \ast)$ is an APA. Thus, the equality \eqref{false_distri} follows from the equation \eqref{p_one} and from \cref{lemma_q1} we know that $(S, \ast) \in \mathcal{V}_{Q_1}$. 
       
    For the converse, equation \eqref{p_one} is obvious. In addition, for $x, y, z \in S$, we have:
    \begin{align*}
           y \ast z = x \ast y  \ast z=x \ast y \ast y \ast z=x \ast y \ast (x \cdot y) \ast z,
       \end{align*}
       i.e., equation \eqref{p_two} is satisfied.
    \end{proof}
\end{prop}

\smallskip

\subsection{Construction of right-trivally distributive APAs}
\medskip

 One can easily note that by (RTD) and \eqref{false_distri} we get the same result as Lemma \ref{lm:form} for left-trivially distributive APAs.

\begin{lemma}\label{lm:form1}
    Let $(S,\cdot,\ast)$ be a right-trivially distributive APA generated by a set $X\subseteq S$. Let $\langle X\rangle_\ast$ be the subsemigroup of $(S,\ast)$ generated by $X$. Then for each $w\in S$ there exist $s_1,\ldots,s_p\in \langle X\rangle_\ast$ such that $w=s_1\cdot s_2\cdot \ldots\cdot s_p$.
    \begin{proof}
    The proof is analogue to the proof of \cref{lm:form}.
    \end{proof}
    \end{lemma}
    
    Let $(S,\ast)\in \mathcal{V}_{Q_1}$ and $(B(S),\cdot)$ be a semigroup generated by $S$. 
For $x= \displaystyle\prod_{i=1}^nx_i$ and $y= \displaystyle\prod_{j=1}^my_j$, with $x_1,\ldots,x_n,y_1,\ldots,y_m\in S$, let us define the following binary operation:
\begin{align*}
    x\otimes y:=(x_n\ast y_1)\cdot\prod_{j=2}^m(y_{j-1}\ast y_j).
\end{align*}
\begin{prop}
    $(B(S),\cdot,\otimes)$ is a right-trivially distributive APA.
    \begin{proof}
  Let $x=\displaystyle\prod_{i=1}^nx_i$, $y=\displaystyle\prod_{j=1}^my_j$, $z=\displaystyle\prod_{k=1}^pz_k\in B(S)$, with $x_1,\ldots,x_n,y_1,\ldots,y_m,z_1,\ldots,z_p\in S$. We will show that the operation $\otimes$ is associative. 
  \begin{align*}
      (x\otimes y)\otimes z
      &=\left((x_n\ast y_1)\cdot \prod_{j=2}^m(y_{j-1}\ast y_j)\right)\otimes \prod_{k=1}^pz_k=((y_{m-1}\ast y_m)\ast z_1)\cdot \prod_{k=2}^p(z_{k-1}\ast z_k)\\
      &=(y_m\ast z_1)\cdot \prod_{k=2}^p(z_{k-1}\ast z_{k})\\
      &=(x_n\ast( y_m\ast z_1))\cdot((y_m\ast z_1)\ast(z_1\ast z_2))\cdot \prod_{k=3}^p\left((z_{k-2}\ast z_{k-1})\ast(z_{k-1}\ast z_k)\right)\\
     &=\prod_{i=1}^nx_i\otimes\left((y_m\ast z_1)\cdot \prod_{k=2}^p(z_{k-1}\ast z_{k})\right)=x\otimes (y\otimes z).
       \end{align*}
  As a consequence, we also have:
  \begin{align*}
      y\otimes z&=
      (y_m\ast z_1)\cdot \prod_{k=2}^p(z_{k-1}\ast z_{k})=((y_{m-1}\ast y_m)\ast z_1)\cdot \prod_{k=2}^p(z_{k-1}\ast z_k)=
      x\otimes y\otimes z,
  \end{align*}
  which shows that $(B(S),\otimes)$ belongs to the variety $\mathcal{V}_{Q_1}$.
Further, \eqref{false_distri} follows by
\begin{align*}
      x\otimes (y\cdot z)=    
      &\left(\prod_{i=1}^nx_i\otimes \prod_{j=1}^my_j\right)\cdot \left(\prod_{j=1}^my_j\otimes \prod_{k=1}^p z_k\right)=(x\otimes y)\cdot (y\otimes z).
  \end{align*}
  By Proposition \ref{prop:rtd}, $(B(S),\cdot,\otimes)$ is right-trivially distributive APA.
    \end{proof}
\end{prop}

\smallskip

Let $X$ be a non-empty set, $\left(\mathcal{B}(X), \, \ast \right)$ the free semigroup in the variety $\mathcal{V}_{Q_1}$ over the set $\mathcal{X}$, $\mathcal{S}$ a variety of semigroups $(S,\cdot)$, and let $F_{\mathcal{S}}\left(\mathcal{B}(X)\right)$ be the free semigroup in $\mathcal{S}$ generated by the set $\mathcal{B}(X)$. 
\begin{theor}
    The algebra $(F_{\mathcal{S}}\left(\mathcal{B}(X)\right),\cdot,\otimes)$ is free in the variety of all right-trivially distributive APAs $(S,\cdot,\ast$) over the set $\mathcal{X}$ such that $(S,\cdot)\in \mathcal{S}$.
    \begin{proof}
        The proof goes exactly in the same way as the proof of Theorem \ref{th:LTDconst}.
    \end{proof}
      \end{theor}

      \smallskip

  \begin{ex}
   Let $(\mathcal{B}(\{x\}),\ast)=(\{x,x^2\},\ast)$ be the $\mathcal{V}_{Q_1}$-semigroup from Example \ref{ex:Q} and consider the free band $\left(S(\mathcal{B}(\{x\})),\cdot\right)$ generated by the set $\{x,x^2\}$. Then $$S(\mathcal{B}(\{x\}))=\{x,x^2,x\cdot x^2, x^2\cdot x, x\cdot x^2\cdot x, x^2\cdot x\cdot x^2\}$$ and $\left(S(\mathcal{B}(\{x\})),\cdot,\otimes\right)$ is an APA with $\theta_a(b)=x^2$, for all $a,b\in S(\mathcal{B}(\{x\}))$.     
  \end{ex}

\smallskip

\section{Some classes of right-trivially distributive APAs} 
In this section, we present some families of right-trivially distributive APAs.

\smallskip

\subsection{APA \texorpdfstring{$(S,\cdot,\ast)$}{} such that \texorpdfstring{$(S, \cdot)$}{} has a right annihilator}
We consider APAs $(S, \cdot, \ast)$ such that the semigroup $(S, \cdot)$ has at least one a right annihilator, i.e. there is an element $R\in S$ such that for all $x\in S$, $x\cdot R=R$. Let
\begin{align*}
\Ann_r(S, \cdot):=\{R \in S\, \mid \,  \,  \text{$R$ is a right annihilator in $(S, \cdot)$}\}.
\end{align*}

\smallskip

We introduce the following preparatory lemma.
\begin{lemma}\label{lemma_r_ann}
    Let $(S,\cdot, \ast)$ be an APA and $R\in \Ann_r(S, \cdot)$. Then for all $x \in S$:
    \begin{enumerate}
      \item $\theta_x=\theta_{R \cdot x}$,
      \item $\theta_x=\theta_R\theta_x\theta_x$,
      \item $\theta_x\theta_R=\theta_R\theta_R$.
  \end{enumerate}   
    \begin{proof}
  Let $x \in S$. Clearly, $R\cdot x, R\in \E(S,\cdot)$. By \cref{lm:cubic}-2., $\theta_{R\cdot x}\underset{ \eqref{rel}}{=} \theta_R\theta_{R\cdot x}\theta_{R\cdot x}=\theta_R\theta_{R}\theta_{R\cdot x}$ and so
\begin{align*}
    \theta_x\underset{\eqref{rel}}{=}\theta_R\theta_x\theta_{R \cdot x}=\theta_R\theta_x\theta_R\theta_R\theta_{R \cdot x}=\theta_R\theta_R\theta_{R \cdot x}=\theta_{R \cdot x}.
\end{align*}
Thus $\theta_x=\theta_R\theta_x\theta_{R \cdot x}=\theta_R\theta_x\theta_x.$
Moreover, by \cref{lm:cubic}-3., $\theta_x\theta_R=\theta_{R \cdot x}\theta_R=\theta_R\theta_R$.
    \end{proof}
\end{lemma}

\smallskip

\begin{theor}
    Let $(S,\cdot, \ast)$ be an APA such that $\Ann_r(S, \cdot)$ is non-empty. Then $(S, \cdot, \ast)$ is right-trivially distributive.
        \begin{proof}
  Let $x, y \in S$ and $R \in \Ann_r(S, \cdot)$. Then we have:
    \begin{align*}
           \theta_x\theta_y\underset{\ref{lemma_r_ann}-1.}{=}\theta_{R \cdot x}\theta_y\underset{\eqref{rel_4}}{=} \theta_x\theta_y\theta_R\theta_{x \cdot y} \underset{\ref{lemma_r_ann}-3.}{=}\theta_x\theta_R\theta_R\theta_{x \cdot y}\underset{\ref{lemma_r_ann}-3.}{=}\theta_R\theta_R\theta_R\theta_{x \cdot y}\underset{\ref{lm:cubic}-1.}{=}\theta_R\theta_{x \cdot y}. 
       \end{align*} 
 Hence,  we obtain    
     $ \theta_y\underset{\eqref{rel}}{=}\theta_x\theta_y\theta_{x \cdot y}=\theta_R\theta_{x \cdot y}\theta_{x \cdot y}\underset{\ref{lemma_r_ann}-2.}{=}\theta_{x \cdot y}$,
thus $(S, \cdot, \ast)$ is RTD. 
    \end{proof}
\end{theor}

\smallskip

\subsection{APA \texorpdfstring{$(S,\cdot,\ast)$}{} such that \texorpdfstring{$(S, \cdot)$}{} has a left identity}

We classify APA $(S, \cdot, \ast)$ such that the semigroup $(S, \cdot)$ has a left identity, i.e., there exists $e\in S$ such that $e \cdot x=x$, for all $x \in S$.

\begin{lemma}\label{lemma_leftid}
    Let $(S,\cdot, \ast)$ be an APA and $e\in S$ a left identity for $(S, \cdot)$. Then, for all $x \in S$, the following hold:
    \begin{enumerate}
    \item $\theta_e\theta_x=\theta_x\theta_x$,
    \item $\theta_x=\theta_e\theta_e\theta_x$,
        \item $\theta_x\theta_e=\theta_e\theta_e$.
    \end{enumerate}
\end{lemma}
\begin{proof}
  Note that 1. and 3. follow directly by Lemma \ref{lm:cubic}, since $e\in \E(S,\cdot)$. 
   Let $x \in S$. 
Then
$\theta_x\underset{\eqref{rel}}{=}\theta_e\theta_x\theta_{x}\underset{1.}{=}\theta_e\theta_e\theta_x,
$
which finishes the proof.
\end{proof}


\begin{theor}
    Let $(S,\cdot, \ast)$ be an APA and assume that $(S, \cdot)$ admits a left identity. Then $(S,\cdot, \ast)$ is right-trivially distributive.
    \begin{proof}
        For $x, y \in S$, we have
           $\theta_x\theta_y=\theta_{e \cdot x}\theta_y \underset{\eqref{rel_4}}{=} \theta_x\theta_y\theta_e\theta_{x \cdot y} \underset{\ref{lemma_leftid}-3.}{=} \theta_x\theta_e\theta_e\theta_{x \cdot y} \underset{\ref{lemma_leftid}-2.}{=} \theta_x\theta_{x \cdot y}.$
As a consequence,  we obtain    
  \begin{align*}
      \theta_y&\underset{\eqref{rel}}{=}\theta_x\theta_y\theta_{x \cdot y}=\theta_x\theta_{x \cdot y}\theta_{x \cdot y}\underset{\ref{lemma_leftid}-1.}{=} \theta_x\theta_{e}\theta_{x \cdot y} \underset{\ref{lemma_leftid}-3.}{=} \theta_e\theta_e\theta_{x \cdot y}=\theta_{x \cdot y},
  \end{align*}
i.e., $(S, \cdot, \ast)$ is RTD. 
    \end{proof}
\end{theor}

\smallskip

\subsection{Some classes of APAs \texorpdfstring{$(S,\cdot,\ast)$}{} such that \texorpdfstring{$(S, \cdot)$}{} is a Clifford semigroup}

\smallskip

Clearly, if $(S, \cdot)$ is a group, by \cref{monoid}, $(S,\ast)$ is determined by $\theta_1$.  
In a more general setting, assume that $(S, \cdot)$ is a Clifford semigroup, that is, a completely regular inverse semigroup. 

In \cite{MaPeSt24}, there is a classification of pentagon algebras $(S, \cdot, \ast)$ arising from Clifford semigroups $(S, \cdot)$ which additionally are \emph{$\E(S, \cdot)$-invariant}, namely $x \ast e= x \ast f$, for all $x\in S$ and $e, f \in \E(S, \cdot)$. We can prove the following.

\begin{prop}
    Let $(S, \cdot)$ be a Clifford semigroup and $(S, \cdot, \ast)$ an $\E(S, \cdot)$-invariant APA.
    Then $(S,\ast)$ is determined by $\theta_e$, with $e \in \E(S, \cdot)$. 
    \begin{proof}
        By \cite[Lemma 12]{MaPeSt24}, $\theta_e=\theta_f$ and $\theta_e\theta_x=\theta_e$, for all $x \in S$ and $e, f \in \E(S, \cdot)$. Hence, 
        $\theta_x\underset{\eqref{rel}}{=}\theta_e\theta_x\theta_{e \cdot x}=\theta_e\theta_{e \cdot x}=\theta_e$, for each $x \in S$.\\
    \end{proof}
\end{prop}

Following \cite[Definition 17]{MaPeSt24}, a pentagon algebra $(S, \cdot, \ast)$ is \emph{$\E(S, \cdot)$-fixed} if $x \ast e=e $, for all $x \in S$ and $e\in \E(S, \cdot)$. In particular, every idempotent of $(S, \cdot)$ is a right-annihilator for $(S, \ast)$.

\begin{theor}
   Let $(S, \cdot)$ be a Clifford semigroup and let $(S, \cdot, \ast)$ be an $\E(S, \cdot)$-fixed APA. Then $(S,\cdot, \ast)$ is right-trivially distributive.
\end{theor}
\begin{proof}
 By the definition, for all $e \in \E(S, \cdot)$ and $x \in S$, 
we have$\theta_e=
\theta_x\theta_e$. Thus,  for $x, y \in S$, we get
$\theta_x\theta_y\underset{\eqref{rel}}{=}  \theta_x\theta_e\theta_y\theta_{e \cdot y}=\theta_e\theta_y\theta_{e \cdot y}=\theta_y$. 
Therefore, by \cref{prop:lzs}, the claim follows.
\end{proof}

\smallskip

\section{Questions and remarks}
In this section, we focus on some open problems and suggest some other APA classes that could be investigated.

\medskip



Let us start by proving the following result which shows that, in the case $(S,  \cdot)$ is a semilattice, then the related APA is both left and right-trivially distributive. We do this since the easy computation will reveal a class of APAs that could be studied.
\begin{prop}\label{semilattice}
Let $(S,\cdot,\ast)$ be an APA such that $(S,\cdot)$ is a semilattice. Then $(S,\ast)$ is determined by some $\gamma\in \End(S,\cdot)$. 
\begin{proof}
By \cref{lm:cubic}-4.,  we have $\theta_x=\theta_{x \cdot y \cdot x}=\theta_{x \cdot y }=\theta_{y \cdot x}=\theta_{y \cdot x \cdot y}=\theta_y$, for all $x \in S$. Thus, the solution belongs to the class of solutions in \cref{ex:gamma}.
\end{proof}
\end{prop}

\smallskip

The proof of \cref{semilattice} suggests to consider
the  class of APAs $(S,\cdot, \ast)$ such that $\theta_x=\theta_{x\cdot y\cdot x}$, for all $x, y \in S$.  By \cref{lm:cubic}-4., obvious examples of such classes are given by APAs $(S,\cdot,\ast)$ with $(S,\cdot)$ being a band. But even such APAs do not have to be neither left nor  right-trivially distributive, as we will see at the end of the section. Therefore the question when APAs satisfying the condition are left or right-trivially distributive naturally arises. At now we have two observations.

\begin{prop}\label{left_triv}
    Let $(S,\cdot, \ast)$ be an APA such that $\theta_x=\theta_{x\cdot y\cdot x}$, for all $x, y \in S$. Then $(S,\cdot, \ast)$ is left-trivially distributive if and only if $\theta_x^2=\theta_y^2$, for all $x, y \in S$.
    \begin{proof}
        First, assume that $(S,\cdot, \ast)$ is left-trivially distributive. Then $\theta_{x \cdot y \cdot x}=\theta_x=\theta_{x \cdot y}$, for all $x, y \in S$. Moreover,
            $\theta_x^2\underset{\eqref{rel}}{=}\theta_y\theta_x\theta_{y \cdot x}\theta_x=\theta_y\theta_x\theta_{y \cdot x}\theta_{x \cdot y \cdot x}\underset{\eqref{rel}}{=}\theta_y\theta_{y \cdot x}=\theta_y^2$.
        
        Vice versa, suppose that $\theta_x^2=\theta_y^2$, for all $x, y \in S$. Thus, by \eqref{rel_4}, we obtain
        \begin{align*}
            \theta_{x \cdot y}\theta_y=\theta_y^2\theta_x\theta_{y \cdot y}=\theta_x^2\theta_x\theta_{y \cdot y}=\theta_x\theta_y^2\theta_{y \cdot y}\underset{\eqref{rel}}{=}\theta_x\theta_y.
        \end{align*}
  Hence, 
            \begin{align}\label{eq_xxyy}
                \theta_{x \cdot y}\underset{\eqref{rel}}{=}\theta_y\theta_{x \cdot y}\theta_{y \cdot x \cdot y  }=\theta_y\theta_{x \cdot y}\theta_{y }=\theta_y\theta_x\theta_y.
            \end{align}      
            The last equality implies that $\theta_{x \cdot y}=\theta_y(\theta_y\theta_x\theta_y)\theta_y=\theta_x^5.$ On the other hand,
          \begin{align*}
            \theta_x\underset{\eqref{rel}}{=}\theta_x^2\theta_{x \cdot x}\underset{\eqref{eq_xxyy}}{=}\theta_x^2\theta_x\theta_x\theta_x=\theta_x^5,
          \end{align*} 
          which concludes the proof.
    \end{proof}
\end{prop}

\smallskip

Consequently, \cref{left_triv} 
gives the following:
\begin{cor}\label{band}
    Let $(S, \cdot, \ast)$ be an APA such that $(S, \cdot)$ is a band.  Then $(S, \cdot, \ast)$ is left-trivially distributive if and only if $\theta_x^2=\theta_y^2$, for all $x, y \in S$.
\end{cor}


\begin{ex} 
Let $n \geq 1$, $S=\{0,1,\ldots,n-1,n\}$, and define for $x,y\in S$ two binary operations:
\begin{align*}
x\cdot y=\begin{cases}x &{\rm if}\quad x\neq n\; \,{\rm or}\,\; x=y=n\\
0 &{\rm if}\quad x=n\; \,{\rm and}\; \,y\neq n
\end{cases}
\quad \text{and} \quad 
x\ast y=\begin{cases}0 &{\rm if}\quad x,y\neq 1\,\, {\rm or}\, \,x=y=1\\
1  &{\rm if}\quad x=1,\; y\neq 1\; \,{\rm or}\,\; x\neq 1,\; y= 1
\end{cases}.
\end{align*}
 Clearly, $(S,\cdot)$ is a band and $(S,\cdot,\ast)$ is an APA such that 
the operation $\ast$ is commutative, and $(S\setminus\{n\},\cdot)$ is a left-zero semigroup.  Moreover, in the semigroup $(S,\cdot)$, the element $0$ is a left, but not right, unique annihilator. Hence, by \cref{thm:leftanih},  
$(S,\cdot,\ast)$  is left-trivially distributive. Note also that for all $k\neq 1$, $\theta_k=\theta_0$ and $\theta_0\theta_0=\theta_0=\theta_1\theta_1$.

\end{ex}

\begin{prop}\label{RTDxyx}
    Let $(S,\cdot, \ast)$ be an APA such that $\theta_x=\theta_{x\cdot y\cdot x}$, for all $x, y \in S$. Then $(S,\cdot, \ast)$ is right-trivially distributive if and only if $\theta_{x \cdot x}=\theta_{y \cdot x}$, for all $x, y \in S$.
\end{prop}
\begin{proof}
    Clearly, if $(S,\cdot, \ast)$ is a right-trivially distributive APA, then, for all $x, y \in S$, $\theta_{x \cdot x}=\theta_x=\theta_{y \cdot x}$. Conversely,
    suppose that $\theta_{x \cdot x}=\theta_{y\cdot x}$, for all $x, y \in S$. 
    Then, 
    applying \eqref{rel} several times, we get
    \begin{align*}
        \theta_{y \cdot x}&=\theta_x\theta_{y \cdot x}\theta_{x \cdot y \cdot x}=\theta_x\theta_{y \cdot x}\theta_x=\theta_x\theta_{y \cdot x}(\theta_y\theta_x\theta_{y \cdot x}){=}\theta_x\theta_{y \cdot x}\theta_y(\theta_y\theta_x\theta_{y \cdot x})\theta_{y \cdot x}\\
        &=\theta_x(\theta_{y \cdot x}\theta_y\theta_{y \cdot x \cdot y})\theta_x\theta_{y \cdot x}\theta_{y \cdot x}{=}\theta_x\theta_y\theta_x\theta_{y \cdot x}\theta_{y \cdot x}=\theta_x\theta_x\theta_{x \cdot x}=\theta_x,
    \end{align*}
    and so $(S,\cdot, \ast)$ is RTD.
\end{proof}
The following example shows that not all APAs with band reduct are left or right trivially distributive.
 \begin{ex}
Let $(S,\ast)$ be a semigroup generated by a set $\{a,b\}$ and such that $a^4=a^2$, $b^2=b$,  $a^2\ast b=b$, $b\ast a^2=a^2$ and $a^2\ast x\ast y=b\ast x\ast y=x\ast y$
for any $x,y\in S$. Then $S=\{a,a^2,a^3,b,a\ast b,b\ast a,a\ast b\ast a\}$. Note that $(S,\ast)$ is not right-normal: $a\ast b\ast a\neq b\ast  a^2=a^2$ and  $\theta_a^2(a)=\theta_a(a^2)=a^3\neq b\ast a=\theta_{b}(b\ast a)=\theta_{b}^2(a)$.  

Now let define two binary operations on the set $S\times S$ in the following way: for $(x_1,x_2), (y_1,y_2)\in S\times S$, 
$(x_1,x_2)\cdot(y_1,y_2):=(x_1,y_2)$ and
\begin{align*}
   &(x_1,x_2)\,\hat{\ast}\,(y_1,y_2):=\begin{cases}
       \left(x_1\ast x_2^2\ast a,x_1\ast a^3\right) &\text{if $y_1=y_2=a$} \\
        \left(x_1\ast  y_1,x_1\ast y_1^2 \ast a\right) &\text{if $y_1\neq a, \; y_2=a$}\\
        \left(x_1\ast x_2^2\ast a,x_1\ast y_2 \right) &\text{if $y_1= a, \; y_2\neq a$}\\
        \left(x_1\ast y_1,x_1\ast y_2 \right) &\text{if $y_1\neq a, \; y_2\neq a$}
   \end{cases}
\end{align*} 

By direct calculations, we obtain that $(S\times S,\cdot,\hat{\ast})$ is an APA with $(S\times S,\cdot)$ being a rectangular band. 
Further, 
$\hat{\theta}_{(a,a)\cdot (b,b)}((a,a))=
(a\ast b\ast a,a^2)\neq 
(a^2,a^2)=
\hat{\theta}_{(a,a)}\left((a,a)\right)$ and $\hat{\theta}_{(a,a)\cdot (b,b)}((a,a))\neq \hat{\theta}_{(b,b)}((a,a))=(b\ast a,b\ast a^3).$
Then 
$\left(S\times S,\cdot,\hat{\ast}\right)$  is neither LTD nor RTD.

\end{ex}

\medskip



\smallskip
The discussion above suggests a potential problem and question.
\begin{prob}
Characterize APAs $(S, \cdot, \ast)$ such that $(x \cdot y \cdot x) \ast z=x \ast z$, for all $x, y, z\in S$, in particular, such that $(S,\cdot)$ is a band.
\end{prob}
\begin{que*}
 Do they exist other natural classes of APAs?  
\end{que*}

\bigskip

\section*{Acknowledgements}
\small{\noindent This work was partially supported by University of Salento - Department of Mathematics and Physics ``E. De Giorgi”, and Warsaw University of Technology - Faculty of Mathematics and Information Science. }\\
M. Mazzotta is a member of GNSAGA (INdAM) and of the nonprofit association ``AGTA-Advances in Group Theory and Applications". M. Mazzotta is supported by AGRI@INTESA- ``NATIONAL
CENTRE FOR HPC, BIG DATA AND QUANTUM COMPUTING", CUP: F83C22000740001.

\section*{Competing interest} \small{\noindent \b{The authors do not have competing interests to declare.
}}

\medskip

\bibliographystyle{elsart-num-sort}  
\bibliography{bibliography}

 
 \end{document}